\documentclass[a4paper,reqno]{amsart}
 \date{24 January 2014}
 \usepackage{amssymb, amsmath, amscd}
\usepackage{color}
\newtheorem{theorem}{Theorem}[section]

\newtheorem{lemma}[theorem]{Lemma}

\newtheorem{remark}[theorem]{Remark}
\newtheorem{corollary}[theorem]{Corollary}

\def\HU#1#2{[#1,#2)}
\def\HL{{[0,\infty)}}\def\IN{{[0,1]}}\def\NHL{{[1,\infty)}}

\def\XZERO{\frac12}
\def\XONE{\frac12}
\def\XTWO{\frac a{4(1-a)}}
\def\XTHREE{\frac1{2(a-1)}}
\def\XFOUR{\frac12}
\def\XFIVE{\frac{a-1}{4(2-a)}}
\def\XSIX{\frac{3-a}{4(1-a)(2-a)}}
\def\XSEVEN{\frac1{(1-a)(2-a)}}
\def\XEIGHT{-\frac{a+1}{4(2-a)(1-a)}}
\def\XNINE{-\frac{1-a}{8(2-a)}}
\def\XTEN{\frac{3a-5}{16(2-a)}}
\def\XELEVEN{-\frac{1-a}{8(2-a)}}
\def\XTWELVE{-\frac{3-a}{4(1-a)(2-a)}}
\def\XTHIRTEEN{0}
\def\XFOURTEEN{\frac1{2(a-2)}}

\makeatletter
 \@addtoreset{equation}{section}
\makeatother
\begin{document}
\title[Heat Content]
{Heat content with singular initial temperature and singular
specific heat}

\author{M. van den Berg and  P. Gilkey}
\address{School of Mathematics, University of Bristol\\
University Walk, Bristol BS8 1TW, UK}~
\begin{email}{mamvdb@bristol.ac.uk}\end{email}
\address {Mathematics Department, University of Oregon\\
Eugene, OR 97403, USA} \begin{email}
{gilkey@uoregon.edu}\end{email}
\begin{abstract} Let $(M,g)$ be a compact Riemannian manifold without boundary.
Let $\Omega$ be a compact subdomain of $M$ with smooth boundary.
We examine the heat content asymptotics for the heat flow
from $\Omega$ into $M$ where both the initial temperature
and the specific heat are permitted
to have controlled singularities on $\partial\Omega$. The
operator driving
the heat process is assumed to be an operator of Laplace type.
 \end{abstract}
\maketitle
\section{Introduction}

The conduction of heat or the diffusion of matter through the
parts of a body is a classical subject in engineering. For an
early reference we refer to the work of Carslaw and Jaeger
\cite{CJ}, and the references therein. From a mathematical point
of view both heat content and heat trace link the spectral
resolution of an operator of Laplace type to the underlying
geometry of the manifold and its boundary in a natural way. The
heat trace shows up when calculating the thermodynamic properties
of quantum systems. See e.g. \cite{B}. A considerable amount of
progress has been made in the last few decades in the
understanding of the short time asymptotic behaviour of the heat
content in a variety of settings and boundary conditions
\cite{PG}. It was discovered by Preunkert \cite{MPPP1,MPPP2,P}
that if we consider the setting where a region $\Omega$ in
Euclidean space $\mathbb{R}^m$ is at initial temperature $1$ while
the complement is at initial temperature $0$ then the integral of
the temperature in $\Omega$ at time $t$ denoted by $Q_{\Omega}(t)$
 satisfies
 $$Q_{\Omega}(t)=|\Omega|-\pi^{-1/2}\mathcal{P}(\Omega)t^{1/2}+o(t),\
 t\downarrow 0,$$
where $|\Omega|$ denotes the measure of $\Omega$, and
$\mathcal{P}(\Omega)$ denotes the perimeter of $\Omega$.
Apparently $\mathbb{R}^m\setminus \Omega$ acts for a very small
time approximately as a $0$ Dirichlet boundary condition. It was
subsequently shown \cite{B13,BG} that many of the asymptotic
properties of $Q_{\Omega}(t)\downarrow 0$ are similar to the
situation where Dirichlet or Neumann boundary conditions on the
boundary of $\Omega$ have been imposed \cite{PG}. This paper
investigates properties of heat flow and heat content in the
setting of singular specific heat, and singular initial
temperature distributions. The latter were first examined in
\cite{vdB,vdBGS,vdBGGK} where either the specific heat or initial
boundary condition were singular but not both. It turns out that
in the absence of boundary conditions but in the presence of
doubly singular data a suitable asymptotic series exist for
$t\downarrow 0$ but that the corresponding existence proof is not
as straightforward as one might expect. Previously, in the
presence of both boundary conditions and doubly singular data, the
existence was part of the hypothesis \cite{vdBGK}. This paper is a
first step towards a proof of an asymptotic series in the doubly
singular setting with boundary conditions. We also note that while
the same collection of local geometric invariants appear as
coefficients in the series, the numerical coefficients have to be
computed by a new collection of special case calculations such as
the interval in $\mathbb{R}$.

Let $(M,g)$ be a compact smooth Riemannian manifold of dimension
$m$ without boundary. Let $D_M$ be an operator of Laplace type on
a smooth vector bundle $V$ over $M$. This means that
\begin{equation}\label{E1.a}
D_M=-\left\{g^{ij}(x)\partial_{x_i}\partial_{x_j}+a_1^i(x)\partial_{x_i}+a_0(x)\right\}
\end{equation}
in a system of local coordinates $(x^1,\dots,x^m)$
and relative to some local frame for $V$.
We adopt the {\it Einstein convention} and sum over repeated indices.
Let $\Omega$ be a subdomain of $M$, i.e. $\Omega$ is an $m$-dimensional
 submanifold of $M$. We assume that $\Omega$ is compact with smooth boundary
 $\partial\Omega$. Let $\phi\in L^1(V|_{\Omega})$ represent the initial temperature and let
$\rho\in L^1(V^*|_{\Omega})$ represent the specific heat.
We let $\phi_\Omega$ and $\rho_\Omega$ be the extension of $\phi$ and $\rho$ to $M$
to be zero on the complement of $\Omega$. The heat content of $\Omega$ in $M$
is given for $t>0$ by:
$$\beta_\Omega(\phi,\rho,D_M)(t):
=\int_M\langle e^{-tD_M}\phi_\Omega,\rho_\Omega\rangle dx$$
where $\langle\cdot,\cdot\rangle$ denotes the natural pairing between $V$ and the dual bundle $V^*$,
where $e^{-tD_M}$ is the fundamental solution of the heat equation for the
operator $D_M$, and
where $dx$ is the Riemannian volume element of $M$. If $K_M(x,\tilde x;t)$ is the
kernel of $e^{-tD_M}$, then:
$$\beta_\Omega(\phi,\rho,D_M)(t)=\int_\Omega\int_\Omega
\langle K_M(x,\tilde x;t)\phi(x),\rho(\tilde x)\rangle dxd\tilde x\,.$$
This quantity has been studied previously where the ambient manifold $M$
was Euclidean space \cite{B13}
and where both $\phi$ and $\rho$ were equal to $1$ on $\Omega$
under the very mild condition
that the characteristic function of $\Omega$ was of bounded variation.
There it was also shown
that the heat content of $\Omega$ in $M$
could be finite for all $t>0$ even if $\Omega$ has infinite
Lebesgue measure.
The situation described below is similar in spirit.

\subsection{Growth assumptions on the initial temperature
and on the specific heat}
Let
\begin{equation}\label{E1.b}
\mathcal{O}:=\{(a,b)\in\mathbb{C}:\Re(a)<1,\ \Re(b)<1,\ a+b\ne1,-1,-3,\dots\}\,.
\end{equation}
Let $\phi$ and $\rho$ be smooth on the interior of $\Omega$.
Let $r(x)$ be the geodesic distance from a point $x$ in $\Omega$
to the boundary $\partial\Omega$. The function $r$
is not be smooth on all of $\Omega$.  However $r$ is smooth if
restricted to a small open neighborhood
$\mathcal{A}$ of $\partial\Omega$ in $\Omega$.
Assume in addition that $r^a\phi\in C^\infty(V|_{\mathcal{A}})$ and
that $r^b\rho\in C^\infty(V^*|_{\mathcal{A}})$. In this setting, we say that
$r^a\phi$ and $r^a\rho$ are {\it smooth near the boundary}. Although in practice, one
is only interested in real $(a,b)$, we shall see that it is convenient to pass
to the complex setting to be able to use analytic continuation.
Since $\Re(a)<1$ and $\Re(b)<1$,
$\phi$ and $\rho$ are in $L^1$.
If $\Re(b)>0$, then $\rho$ has a controlled
blowup while if $\Re(b)<0$, then $\rho$ has a
controlled decay as we approach $\partial\Omega$ from within $\Omega$.
If $b=0,-1,-2,\dots$, then in fact $\rho$ is smooth on all of
$\Omega$ when $\Omega$ is regarded as a compact manifold with
boundary in its own right. Viewed as a function on all of $M$, of course,
$\rho_\Omega$ is not smooth on $\partial\Omega$ since $\rho_\Omega$ is zero
on $\Omega^c$. If $k$ is a non-negative integer with $\Re(b)<-k$,
then $\rho_\Omega$ is $C^k$ on $M$.
The situation is similar for $\phi$.
\subsection{The asymptotic series for the heat content}
The following is one of the two main
results of this paper:

\begin{theorem}\label{T1.1} Let $D_M$
be an operator of Laplace type on a smooth vector bundle $V$
over a compact Riemannian manifold $(M,g)$ without boundary.
Let $\Omega$ be a compact subdomain of $M$ with smooth
boundary.
Let $\phi\in C^\infty(V|_{\operatorname{int}(\Omega)})$ and let
$\rho\in C^\infty(V^*|_{\operatorname{int}(\Omega)})$.
Let $(a,b)\in\mathcal{O}$.
We assume that $r^a\phi$ and $r^b\rho$ are smooth near the boundary of $\Omega$.
Let $\beta_\Omega(\phi,\rho,D_M)(t)$ be the
heat content of $\Omega$ in $M$. Then there is a complete
asymptotic expansion of $\beta_\Omega(\phi,\rho,D_M)(t)$
for small time such that for any positive integer $N$ as $t\downarrow0$ we have:
$$\beta_\Omega(\phi,\rho,D_M)(t)=
\sum_{n=0}^Nt^n\beta_{n,a,b}^\Omega(\phi,\rho,D_M)
+\sum_{j=0}^N t^{(1+j-a-b)/2}
   \beta_{j,a,b}^{\partial
   \Omega}(\phi,\rho,D_M)+O(t^{(N-1)/2})\,.$$
The coefficient $\displaystyle\beta_{0,a,b}^{\partial\Omega}(\phi,\rho,D_M)$ of $t^{(1+j-a-b)/2}$
is given by
\begin{eqnarray*}
&&\beta_{0,a,b}^{\partial\Omega}(\phi,\rho,D_M)
=c(a,b)\int_{\partial\Omega}\langle\phi_0,\rho_0\rangle dy,\quad\text{ where}\\
 &&c(a,b):=2^{-a-b}\frac1{\sqrt\pi}\Gamma\left(\frac{2-a-b}2\right)
 \Gamma(a+b-1) \cdot\left(\frac{\Gamma(1-a)}{
 \Gamma(b)}+\frac{\Gamma(1-b)}
{ \Gamma(a)}\right)\,.
\end{eqnarray*}
More generally, the coefficients $\beta_{n,a,b}^\Omega(\cdot)$ are given as regularized integrals of
local invariants over the interior of $\Omega$ that are bilinear in the
derivatives of $\{\phi,\rho\}$ up to order $2n$ with coefficients that depend holomorphically
on the parameters $(a,b)$, that
 depend smoothly on the $0$-jets of the metric, and
that are polynomial in the derivatives of the total symbol of $D_M$ up to order $2n$.
The coefficients $\beta_j^{\partial\Omega}(\cdot)$ are given similarly
as integrals of local invariants over the boundary $\partial\Omega$ where
the derivatives of $\{\phi,\rho\}$ and of the total symbol of $D_M$ are up to order $j$.
\end{theorem}

If the boundary of $\Omega$ is not connected, then we can allow different values of $(a,b)$ (i.e.
different rates
of growth or decay) on the different
components of $\partial\Omega$ and there are separate asymptotic series arising from
each component of the boundary.

 \subsection{Log terms in the asymptotic expansion}\label{S1.3}
 The assumption that $a+b\ne1,-1,-3,\dots$ in the definition of $\mathcal{O}$
 which was given in Equation~(\ref{E1.b}) is essential in Theorem~\ref{T1.1};
 there can be logarithmic
singularities when $a+b=1,-1,-3,\dots$.
Let $K_{\mathbb{R}}(x,\tilde x;t)$ be the heat kernel of $D_{\mathbb{R}}:=-\partial_x^2$ on
${\mathbb{R}}$:
\begin{equation}\label{E1.c}
K_{\mathbb{R}}(x,\tilde x;t):=\frac1{\sqrt{4\pi t}}e^{-(x-\tilde x)^2/(4t)}\,.
\end{equation}
If $\rho$ and $\phi$ belong to $L^1(\Omega)$, then the heat content of $[0,1]$ in $\mathbb{R}$
is given by:
$$
\beta_{\IN}(\phi,\rho,D_{\mathbb{R}})(t):
=\int_0^1\int_0^1K_{\mathbb{R}}(x,\tilde x;t)\phi(x)\rho(\tilde x)x^{-a}\tilde x^{-b}dxd\tilde x\,.
$$
The following is the second main result of this paper:
\begin{theorem}\label{T1.2}
Let $(a,b)\in\mathbb{R}^2$ with $a<1$ and $b<1$.
 Assume that $a+b=1$.
Let $\Xi_i$ be smooth monotonically decreasing
cut-off functions which are identically 1 near $x=0$ and
identically $0$ in an open neighborhood of the interval $[1/2,1]$. Then for $t\downarrow0$
$$\beta_\IN(x^{-a}\Xi_1,x^{-b}\Xi_2,D_\mathbb{R})(t)=\beta_0(a,b,\Xi_1,\Xi_2)-\frac12\log(t)+O(t^{\frac12}\log(t))$$
where for any $\epsilon>0$ sufficiently small,
\begin{eqnarray*}
&&\beta_0(a,b,\Xi_1,\Xi_2)=\frac12\log(\epsilon^2)+\frac12\gamma+\log(2^{1/2}-1)+2\log 2
+\int_{[\epsilon,\frac12]}\Xi_1(x)\Xi_2(x)x^{-1}dx\\
&&\qquad\qquad\quad+\frac12\int_{[0,1]}dqq^{-1}
\left(\frac{(1+q)^{a-1}}{(1-q)^{a}}+\frac{(1-q)^{a-1}}{(1+q)^{a}}-\frac{2}{(1+q^2)^{1/2}}\right)\,.
\end{eqnarray*}
\end{theorem}

The parameter $\epsilon$ serves to regularize the integral and does not contribute to
the value of $\beta_0$. We refer to Theorem~\ref{T4.4} for additional information concerning
the log terms in a more general context.

\subsection{A more general setting}\label{S1.4}
 Let $D_M$ be the Laplace-Beltrami operator.
We have defined the heat content of $\Omega$ in $M$ if the ambient manifold $M$
is a compact smooth Riemannian manifold without boundary.
But it can also be defined using the Dirichlet realization of $D_M$ if $M$ is a smooth
compact Riemannian manifold with boundary.
And it can be defined if $M$ is a complete Riemannian manifold such that
the Ricci curvature of $M$
is bounded from below. Theorem~\ref{T1.3} below was proved in \cite{BG}.
It will let us pass between regarding $\Omega$ as
a subdomain of a compact manifold without boundary,
as a compact subdomain of a manifold with boundary, or a compact
subdomain of a  complete manifold under very mild conditions.
Let $\Omega\subset\operatorname{Interior}\{\tilde M\}\subset\tilde M\subset M$
where $\Omega$ and $\tilde M$ are compact subdomains with smooth boundaries
of a Riemannian manifold $(M,g)$.
Let  $\epsilon:=\operatorname{dist}_g(\partial\Omega,\partial \tilde M)>0$.
Let $\Delta_M$ be the Laplace-Beltrami operator
 on $M$. Let $\Delta_{\tilde M}$ be the Dirichlet realization of $D_M$ on $\tilde M$.
\begin{theorem}\label{T1.3}
Assume that $(M,g)$ is complete with non-negative Ricci curvature.
Let $\rho$ be continuous on $M$ and let $\phi\in L^1(\Omega)$. Then:
$$
|\beta_\Omega(\phi,\rho,\Delta_{\tilde M})(t)-
\beta_\Omega(\phi,\rho,\Delta_M)(t)|\le 2^{(2+m)/2}
\lVert\phi\rVert_{L^1(\Omega)}\lVert\rho\rVert_{L^{\infty}(\Omega)}e^{-\epsilon^2/(8t)}\,.
$$
\end{theorem}

\subsection{Previous results}
It is worth presenting an example to illustrate the kinds of coefficients
which arise in the asymptotic series of Theorem~\ref{T1.1}. We
first introduce some additional notation.
If $D_M$ is an operator of Laplace type, then
there is a unique connection $\nabla$ on $V$ and a unique endomorphism $E$ of $V$
so that we can express $D_M$ in a Bochner formalism:
$$D_M=-(g^{ij}\nabla_{\partial x_i}\nabla_{\partial x_j}+E)\,.$$
Near the boundary, we introduce local coordinates $(r,y)$ so that the curves
$\sigma_t:t\rightarrow(t,y)$
are unit speed geodesics which are perpendicular to the boundary when $t=0$.
Assume that $b=0$ so $\rho$ is smooth on $\Omega$. Near the boundary, we may expand
$$
\phi(r,y)\sim\sum_{i=0}^\infty r^{i-a}\phi_i(y)\text{ and }\rho(r,y)\sim\sum_{j=0}^\infty r^j\phi_j
$$
in modified Taylor series. We use the connection $\nabla$ on $V$ and the dual
connection $\nabla^*$ on $V^*$ to introduce local
frames for $V$ and $V^*$ which are parallel along the geodesics $\sigma_t$.
Let $L_{uv}$ be the components of the second fundamental form relative to a local
orthonormal frame for the tangent bundle of the boundary,
 let $\operatorname{Ric}_{mm}$ be the Ricci curvature of the normal vector field,
 and let $\tau$ be the scalar curvature. Let ``$;u$" denote the components
of covariant differentiation with respect to $\nabla$ on $V$ and with respect to
$\nabla^*$ on $V^*$. The following result \cite{BG} will be an essential ingredient in
 the proofs that we shall give of Theorem~\ref{T1.1} and of
Theorem~\ref{T1.2} subsequently:
\begin{theorem}\label{T1.4}
Adopt the notation of Theorem \ref{T1.1}. Let $b=0$.
There is a complete
asymptotic series as $t\downarrow0$ of the form:
\begin{eqnarray*}
\beta_\Omega(\phi,\rho,D_M)(t)&\sim&\sum_{n=0}^\infty
t^{n}\beta_{n}^\Omega(\phi,\rho,D_M) +\sum_{j=0}^\infty
t^{(1+j-a)/2}\beta_{j,a}^{\partial\Omega}(\phi,\rho,D_M),
\end{eqnarray*}
where $\beta_{n}^\Omega$ and $\beta_{j,a}^{\partial\Omega}$
are integrals of certain locally computable invariants over
$\Omega$ and $\partial\Omega$, respectively, and where the dependence
on $(a,b)$ is holomorphic. In particular, we have:
\ \medbreak
$\beta_{0,a}^{\partial\Omega}
 (\phi,\rho,{D_M})=2c(a,0)\displaystyle\int_{\partial\Omega}
 \XZERO\langle\phi_0,\rho_0\rangle dy$,
\medbreak$ \beta_{1,a}^{\partial\Omega} (\phi,\rho,D_M) =\displaystyle
2c(a-1,0)\int_{\partial\Omega}\left\{\XONE\langle\phi_1,\rho_0\rangle
+\XTWO\langle L_{uu}\phi_0,\rho_0\rangle +\XTHREE
\langle\phi_0,\rho_1\rangle\right\} dy$,
\medbreak
$\beta_{2,a}^{\partial\Omega} (\phi,\rho,{D_M})
 =\displaystyle 2c(a-2,0)\int_{\partial\Omega}\left\{\XFOUR
 \langle\phi_2,\rho_0\rangle
+\XFIVE\langle
L_{uu}\phi_1,\rho_0\rangle
 +\XSIX\langle E
 \phi_0,\rho_0\rangle\right.$
 \medbreak\quad\qquad\qquad
$\displaystyle+\XSEVEN\langle
\phi_0,\rho_2\rangle
\XEIGHT\langle
 L_{uu}\phi_0,\rho_1\rangle
 \XNINE\langle
 \operatorname{Ric}_{mm}\phi_0,\rho_0\rangle$
 \medbreak\quad\qquad\qquad
$\displaystyle +\XTEN\langle
 L_{uu}L_{vv}\phi_0,\rho_0\rangle
 \XELEVEN\langle
L_{uv}L_{uv}\phi_0,\rho_0\rangle$
 \medbreak\quad\qquad\qquad
$\displaystyle\left.
\XTWELVE
\langle \phi_{0;u},\rho_{0;u}\rangle
+\XTHIRTEEN
\langle\tau\phi_0,\rho_0\rangle+
\XFOURTEEN
\langle\phi_1,\rho_1\rangle\right\}dy$.
\end{theorem}

\begin{remark}\label{R1.4}
\rm We note that if $b=-1,-2,\dots$, then $\rho$ is smooth on $\Omega$ and this result
continues to hold. Since $\rho_j=0$ for $j=0,\dots,-b+1$ in this setting, then
we can re-index the series
for the boundary invariants to involve powers $t^{(1+j-a-b)/2}$. This establishes
Theorem~\ref{T1.1} if $b=0,-1,-2,\dots$; again if the boundary of $\Omega$ contains
several components, we can permit $(a,b)$ to vary with the component and add the resulting
boundary contributions.
\end{remark}

\subsection{Outline of the paper}
In Section~\ref{S2}, we shall establish Theorem~\ref{T1.2}. The
proof of Theorem~\ref{T1.1} is considerably more lengthy. Let
$(a,b)\in\mathcal{O}$ where $\mathcal{O}$ is defined in
Equation~\ref{E1.b}. In Section~\ref{S3}, we establish the
existence of an asymptotic series of the form given in
Theorem~\ref{T1.1} for the special case that $\Omega=\IN$ is
embedded in $\mathbb{R}$, that $\phi(x)=x^{-a}$, and that
$\rho(x)=x^{-b}$; see Theorem~\ref{T3.1} for details. In
Section~\ref{S4}, we use the pseudo-differential calculus
depending upon a complex parameter which was developed by Seeley
\cite{S66,S69} to complete the proof of Theorem~\ref{T1.1}; the
special case computation of Theorem~\ref{T3.1} being an essential
input. We shall use the treatment of the pseudo-differential
calculus given in \cite{G94}. Our discussion follows closely that
of \cite{BG} although there are some important differences. In
Section~\ref{S4.5} (see Theorem~\ref{T4.4}) we generalize
Theorem~\ref{T1.1} and Theorem~\ref{T1.2} to establish the
existence of a complete asymptotic series with log terms when
$a+b=1-2k$, $a<1$, $b<1$, and $(a,b)\in\mathbb{R}^2$. We conclude
the paper in Section~\ref{S5} by using the methods of this paper
to study the heat content with Neumann and Dirichlet boundary
conditions for the operator $-\partial_x^2$ on the interval.

\section{The proof of Theorem~\ref{T1.2} -- log terms}\label{S2}

To prove Theorem~\ref{T1.2}, we assume that $a+b=1$ and adapt an
argument of \cite{BG12} to obtain a precise formula for the first
2 terms in the asymptotic series. Let $(a,b)\in\mathbb{R}^2$ with
$a<1$, $b<1$, and $a+b=1$. We choose the notation so $0<b\le a<1$.
We shall apply an argument from \cite{BG12}. Let
$0<\epsilon_1<\epsilon_3<1/2$, and let
$0<\epsilon_2<\epsilon_4<1/2$. We assume that $\Xi_1$ (resp.
$\Xi_2$) is a smooth monotonically decreasing function on $\IN$
which is identically $1$ for $0\le x\le\epsilon_1$ (resp. $0\le
x\le\epsilon_2$) and which vanish identically for
$\epsilon_3\le x$ (resp. $\epsilon_4\le x$). Let

\begin{eqnarray*}
&&C=\{(x,\tilde x)\in \mathbb{R}: x^2+\tilde x^2\ge \epsilon ^2,0\le x\le
\epsilon_3,0\le \tilde x \le \epsilon_4\},\\
&&C_1=\{(x,\tilde x)\in C: |x-\tilde x|\le \sigma\},
\end{eqnarray*}
where $\sigma \in (0,\epsilon/5)$ will be chosen later on.
 We decompose the heat content as follows.
\begin{equation}\label{E2.a}
\begin{array}{l}
\beta_\IN(x^{-a}\Xi_1,x^{-b}\Xi_2,D_{\mathbb{R}})(t)=B_1(t)+B_2(t),\text{ where}\\
B_1(t):=\displaystyle\iint_{\mathbb{R}_+^2\cap\{x^2+\tilde x^2<\epsilon^2\}}K_{\mathbb{R}}(x,\tilde x;t)x^{-a}\tilde x^{-b}dxd\tilde x,\\
B_2(t):=\displaystyle\vphantom{\int^{X^X}}\iint_{C}K_{\mathbb{R}}(x,\tilde x;t)\Xi_1(x)\Xi_2(\tilde x)x^{-a}\tilde x^{-b}dxd\tilde x.
\end{array}\end{equation}

To estimate $ B_2(t)$ we first consider the contribution from the set
$C\setminus C_1$. We have that
$$
K_{\mathbb{R}}(x,\tilde x;t)\le t^{-1/2}e^{-(x-\tilde x)^2/(4t)}\le
t^{-1/2}e^{-\sigma^2/(4t)}, (x,\tilde x)\in C\setminus C_1.
$$
Consequently
\begin{equation}\label{E2.b}
\begin{array}{l}
\iint_{C\setminus
C_1}K_{\mathbb{R}}(x,\tilde x;t)\Xi_1(x)\Xi_2(\tilde x)x^{-a}\tilde x^{-b}dxd\tilde x\\
\qquad \le Kt^{-1/2}e^{-\sigma^2/(4t)},\text{ where }
K=\iint_C\Xi_1(x)\Xi_2(\tilde x)x^{-a}\tilde x^{-b}dxd\tilde x\,.
\vphantom{\vrule height
14pt}\end{array}
\end{equation}
On the set $C\cap \{|x-\tilde x|\le \epsilon/5\}$ we have that
$\tilde x\rightarrow \Xi_2(\tilde x)\tilde x^{-b}$ is $C^{\infty}$. Hence
there exists $L=L(\epsilon,b,\Xi_2)$
such that
$|\Xi_2(\tilde x)\tilde x^{-b}-\Xi_2(x)x^{-b}|\le
L|x-\tilde x|$. Because $x\ge \epsilon/2$ and
$\tilde x\ge \epsilon/2$ on $C\cap \{|x-\tilde x|\le \epsilon/5\}$ and because
$K_{\mathbb{R}}(x,\tilde x;t)\le t^{-1/2}$, we have that
\begin{equation}\label{E2.c}
\begin{array}{l}\displaystyle
\iint_{C_1}K_{\mathbb{R}}(x,\tilde x;t)\Xi_1(x)x^{-a}L|x-\tilde x|dxd\tilde x
\le L(2/\epsilon)^{a}t^{-1/2}\iint_{C_1}|x-\tilde x|dxd\tilde x\\
\qquad\qquad\qquad\qquad\qquad\qquad\qquad\qquad\hphantom{aa}
\le  2 L(2/\epsilon)^{a}t^{-1/2}\sigma^2.
\vphantom{\vrule height 12pt}\end{array}\end{equation}
We now choose $\sigma^2$ as to minimize
$t^{-1/2}e^{-\sigma^2/(4t)}+t^{-1/2}\sigma^2$ by taking:
$$
\sigma^2=4t\log\left( \frac{1}{4t}\right).
$$
This gives that for $t$ sufficiently small, the right hand sides of
Equation~(\ref{E2.b}) and of Equation~(\ref{E2.c}) are $O(t^{1/2})$ and
$O(t^{1/2}\log(t^{-1}))$, respectively. We conclude that
\begin{equation}\label{E2.d}
B_2(t)=\iint_{C_1}
K_{\mathbb{R}}(x,\tilde x;t)\Xi_1(x)\Xi_2(x)x^{-1}dxd\tilde x+O(t^{1/2}\log(t^{-1})).
\end{equation}
We now write
$$
C_1=(C_1\cap \{x^2\ge \epsilon^2/2\})\cup (C_1 \cap\{x^2<
\epsilon^2/2\})=C_2 \cup C_3.
$$
Since $x\ge \epsilon/2$ on $C_1$, the integrand in the first term in the right hand side of
Equation~(\ref{E2.d}) is bounded by $2\epsilon^{-1}t^{-1/2}$. Hence
$$
\iint_{C_3}K_{\mathbb{R}}(x,\tilde x;t)\Xi_1(x)\Xi_2(x)x^{-1}dxd\tilde x\le
2\epsilon^{-1}t^{-1/2}|C_3|,
$$
where $|\cdot|$ denotes Lebesgue measure. Because $|C_3|\le \sigma^2/2$,
$$
0\le\iint_{C_3}K_{\mathbb{R}}(x,\tilde x;t)\Xi_1(x)\Xi_2(x)x^{-1}dxd\tilde x\le
\epsilon^{-1}t^{-1/2}\sigma^2,
$$
and so the contribution from $C_3$ to the integral in
Equation~(\ref{E2.d}) is $O(t^{1/2}\log(t^{-1}))$. Furthermore
\begin{eqnarray*}
&&\iint_{C_2}K_{\mathbb{R}}(x,\tilde x;t)\Xi_1(x)\Xi_2(x)x^{-1}dxd\tilde x\\
&\le& \iint_{\{x^2\ge
 \epsilon^2/2\}}K_{\mathbb{R}}(x,\tilde x;t)\Xi_1(x)\Xi_2(x)x^{-1}dxd\tilde x \\
&=&\int_{[\epsilon/\sqrt2, 1/2]}\Xi_1(x)\Xi_2(x)x^{-1}dx.
\end{eqnarray*}
To obtain a lower bound for the contribution from $C_2$ to the
integral in Equation~(\ref{E2.d}), we note that
\begin{eqnarray*}
&&\iint_{C_2}K_{\mathbb{R}}(x,\tilde x;t)\Xi_1(x)\Xi_2(x)x^{-1}dxd\tilde x
\vphantom{\vrule height 12pt} \\
&\ge&\iint_{\{x^2\ge \epsilon^2/2\}}K_{\mathbb{R}}(x,\tilde x;t)\Xi_1(x)\Xi_2(x)x^{-1}dxd\tilde x\\
&&\quad -\iint_{\{|x-\tilde x|\ge\sigma \}\cap \{x^2\ge
\epsilon^2/2\}}K_{\mathbb{R}}(x,\tilde x;t)
\Xi_1(x)\Xi_2(x)x^{-1}dxd\tilde x\\
  &=&\int_{[\epsilon/\sqrt2, 1/2]}\Xi_1(x)\Xi_2(x)x^{-1}dx\\
&&\quad -\iint_{\{|x-\tilde x|\ge\sigma \}\cap \{x^2\ge
\epsilon^2/2\}}K_{\mathbb{R}}(x,\tilde x;t)\Xi_1(x)\Xi_2(x)x^{-1}dxd\tilde x.
\end{eqnarray*}
Moreover
\begin{eqnarray*}
\int_{\{|x-\tilde x|\ge \sigma\}}K_{\mathbb{R}}(x,\tilde x;t)d\tilde x&\le&
\int_{\{|x-\tilde x|\ge \sigma\}}(4\pi t)^{-1/2}
e^{-|x-\tilde x|\sigma/(4t)}d\tilde x
\nonumber \\
&=&4\pi^{-1/2}t^{1/2}\sigma^{-1}e^{-\sigma^2/(4t)}=O(t^2).
\end{eqnarray*}
Because $\Xi_1(x)\Xi_2(x)x^{-1}=x^{-1}$ for $0<x\le \epsilon$, we have:
\begin{eqnarray*}
 B_2(t)&=&\int_{[\epsilon/\sqrt2, 1/2]}\Xi_1(x)\Xi_2(x)x^{-1}dx+O(t^{1/2}\log(t^{-1}))\\
&=&\int_{[\epsilon, 1/2]}\Xi_1(x)\Xi_2(x)x^{-1}dx+2^{-1}\log
2+O(t^{1/2}\log(t^{-1})),
\end{eqnarray*}
In order to obtain the asymptotic behaviour of $B_1$ in
Equation~(\ref{E2.a}), we introduce polar coordinates $x=(4t)^{1/2}\varrho \cos
\theta, y=(4t)^{1/2}\varrho \sin \theta$ to find that
$$ B_1(t)=\pi^{-1/2}\int_{[0,\pi/2]}d\theta(\cos \theta)^{-a}(\sin
\theta)^{a-1}\int_{[0,\epsilon/(4t)^{1/2}]}d\varrho
e^{-\varrho^2(1-\sin(2\theta))}.$$
A further change of variable $ \theta=\Phi+\pi/4$ yields that
\begin{eqnarray*}
B_1 (t)&=&(2/\pi)^{1/2}\int_{[0,\pi/4]}d\Phi\left(\frac{(\cos\Phi+\sin
\Phi)^{a-1}}{(\cos\Phi-\sin
\Phi)^{a}}+\frac{(\cos\Phi-\sin
\Phi)^{a-1}}{(\cos\Phi+\sin \Phi)^{a}}\right)\\
& & \ \times\int_{[0,\epsilon/(4t)^{1/2}]}d\varrho e^{-2\varrho^2(\sin
\Phi)^2}= B_3(t)+B_4(t)+B_5(t)\,.\end{eqnarray*}
We make the change of variables $\tan \Phi=q$ to see:
\medbreak\quad\
$\displaystyle B_3 (t)$
\medbreak\qquad\
$\displaystyle=(2/\pi)^{1/2}\int_{[0,\pi/4]}d\Phi\left(\frac{(\cos\Phi+\sin
\Phi)^{a-1}}{(\cos\Phi-\sin
\Phi)^{a}}+\frac{(\cos\Phi-\sin
\Phi)^{a-1}}{(\cos\Phi+\sin \Phi)^{a}}-2\right)
\int_{\HL}d\varrho e^{-2\varrho^2(\sin \Phi)^2}$
\medbreak\qquad\ \
$\displaystyle=2^{-1}\int_{[0,\pi/4]}d\Phi(\sin\Phi)^{-1}
\left(\frac{(\cos\Phi+\sin \Phi)^{a-1}}{(\cos\Phi-\sin
\Phi)^{a}}+\frac{(\cos\Phi-\sin
\Phi)^{a-1}}{(\cos\Phi+\sin
\Phi)^{a}}-2\right)$
\medbreak\ \ \qquad
$\displaystyle=2^{-1}\int_{[0,1]}dqq^{-1}
\left(\frac{(1+q)^{a-1}}{(1-q)^{a}}+\frac{(1-q)^{a-1}}{(1+q)^{a}}-\frac{2}{(1+q^2)^{1/2}}\right)$.
\medbreak\noindent We also compute that
\medbreak\quad
$\displaystyle B_4 (t)=-(2/\pi)^{1/2}\int_{[0,\pi/4]}d\Phi\left(\frac{(\cos\Phi+\sin
\Phi)^{a-1}}{(\cos\Phi-\sin
\Phi)^{a}}+\frac{(\cos\Phi-\sin
\Phi)^{a-1}}{(\cos\Phi+\sin \Phi)^{a}}-2\right)$
\medbreak$\displaystyle\quad\qquad\qquad\qquad\qquad\qquad
\times\int_{\HU{\epsilon/(4t)^{1/2}}{\infty}}d\varrho
e^{-2\varrho^2(\sin \Phi)^2}$,
\medbreak\noindent and that
\begin{equation}\label{E2.e}
B_5 (t)=(8/\pi)^{1/2}\int_{[0,\pi/4]}d\Phi\int_{[0,\epsilon/(4t)^{1/2}]}d\varrho
e^{-2\varrho^2(\sin \Phi)^2}.
\end{equation}

The expression under \eqref{E2.e} equals
\begin{equation}\label{E2.f}
\begin{array}{ll}
&\ \ \
(8/\pi)^{1/2}\int_{[0,\pi/4]}d\Phi\int_{[0,\epsilon/(4t)^{1/2}]}d\varrho
e^{-2\varrho^2(\sin \Phi)^2}\vphantom{\vrule height 12pt} \\ &
=(8/\pi)^{1/2}\int_{[0,\pi/4]}d\Phi\int_{[0,\epsilon/(4t)^{1/2}]}d\varrho
e^{-2\varrho^2\Phi^2}\\& \ \ \ +\int_{[0,\pi/4]}d\Phi((\sin\Phi)^{-1}-\Phi^{-1})\vphantom{\vrule height 12pt}
\\ & \ \ \
+(8/\pi)^{1/2}\int_{[0,\pi/4]}d\Phi\int_{\HU{\epsilon/(4t)^{1/2}}{\infty}}d\varrho\left(e^{-2\varrho^2\Phi^2}
-e^{-2\varrho^2(\sin\Phi)^2}\right).
\vphantom{\vrule height 12pt}\end{array}\end{equation}
The third term in the right hand side of Equation~(\ref{E2.f}) is
$O(e^{-\epsilon^2/(5t)})$. The second term in the right hand side
of \eqref{E2.f} is equal to $\log(2^{1/2}-1)+3\log 2-\log \pi$. The
first term in the right hand side of \eqref{E2.f} equals
\begin{eqnarray*}
&&
(4/\pi)^{1/2}\int_{[0,\pi\epsilon/(32t)^{1/2}]}d\Phi\int_{[0,\Phi]}d\varrho
e^{-\varrho^2}\\
&=&(4/\pi)^{1/2}\log\left(\frac{\pi\epsilon}{\sqrt{32t}}\right)\int_{[0,\pi\epsilon/(32t)^{1/2}]}d\varrho
e^{-\varrho^2}
-\frac2{\sqrt\pi}\int_{[0,\pi\epsilon/(32t)^{1/2}]}d\Phi(\log
\Phi)e^{-\Phi^2}\\
&=&2^{-1}\log(\epsilon^2/t)+\log\pi+2^{-1}\gamma-3\cdot2^{-1}\log
2+O(e^{-\epsilon^2/(5t)}),
\end{eqnarray*}
where we have used Equation~(4.333) in \cite{grad65} together with
$$
\int_{[0,\pi\epsilon/(32t)^{1/2}]}d\Phi(\log
\Phi)e^{-\Phi^2}=\int_{\HL}d\Phi(\log
\Phi)e^{-\Phi^2}+O(e^{-\epsilon^2/(5t)}).
$$
We find that
$$
B_5 (t)=2^{-1}\log(\epsilon^2/t)+2^{-1}\gamma+\log(2^{1/2}-1)+3\cdot2^{-1}\log
2+O(e^{-\epsilon^2/(5t)}).
$$

In order to estimate $B_4 (t)$ we first note that by expanding
$\sin\Phi$ and $\cos\Phi$ around $0$ we have that
$$
\frac{(\cos\Phi+\sin \Phi)^{a-1}}{(\cos\Phi-\sin
\Phi)^{a}}+\frac{(\cos\Phi-\sin
\Phi)^{a-1}}{(\cos\Phi+\sin \Phi)^{a}}-2=O(\Phi^2).
$$
Furthermore for $\Phi \in [0,\pi/4]$,
\begin{eqnarray*}
0&\le& \int_{\HU{\epsilon/(4t)^{1/2}}{\infty}}d\varrho e^{-2\varrho^2(\sin
\Phi)^2}\vphantom{\vrule height 12pt}
\le(4t)^{1/2}\epsilon^{-1}\int_{\HU{\epsilon/(4t)^{1/2}}{\infty}}d\varrho\varrho
e^{-2\varrho^2(\sin
\Phi)^2}\vphantom{\vrule height 12pt} \\
&\le&(4t)^{1/2}\epsilon^{-1}\int_{\HL}d\varrho\varrho
e^{-2\varrho^2(\sin
\Phi)^2}\vphantom{\vrule height 12pt}
=2^{-1}t^{1/2}\epsilon^{-1}(\sin\Phi)^{-2}.
\end{eqnarray*}
Hence
\begin{eqnarray*}
&&|B_4 (t)|\le
(2\pi)^{-1/2}t^{1/2}\epsilon^{-1}\int_{[0,\pi/4]}d\Phi(\sin\Phi)^{-2}
\times\left|\frac{(\cos\Phi+\sin \Phi)^{a-1}}{(\cos\Phi-\sin
\Phi)^{a}}+\frac{(\cos\Phi-\sin
\Phi)^{a-1}}{(\cos\Phi+\sin
\Phi)^{a}}-2\right|.\nonumber
\end{eqnarray*}
We see that the integral with respect to $\Phi$ converges both at
$\Phi=0$ and at $\Phi=\pi/4$. We conclude that $B_4=O(t^{1/2})$.
This gives the formula of Theorem~\ref{T1.2} when $a+b=1$.

\section{The 1-dimensional setting: $\phi(x)=x^{-a}$ and $\rho(x)=x^{-b}$}\label{S3}
Let $\Re(a)<1$ and $\Re(b)<1$.
Let $h_{a,b}^\IN(t)$ be the heat content of $[0,1]$ in $\mathbb{R}$ where
$\phi(x)=x^{-a}$ and $\rho(x)=x^{-b}$. By Equation~(\ref{E1.c}):
\begin{equation}\label{E3.a}
h_{a,b}^\IN(t):=\frac1{\sqrt{4\pi t}}\int_0^1\int_0^1e^{-(x-\tilde x)^2/(4t)}x^{-a}\tilde x^{-b}dxd\tilde x\,.
\end{equation}
 Let $0<t<1$, let $N$ be a non-negative integer with $N\ge-\Re(a+b)+3$, and
let $0<\delta<\frac12$. Let $\kappa(\cdot)$ be
a bound which depends on the parameters indicated and which may increase
a finite number of times in any given proof. Let $\mathcal{O}$ be as in Equation~(\ref{E1.b})
and let $c(a,b)$ be as in Theorem~\ref{T1.1}.
Section~\ref{S3} is devoted to
the proof of the following special case of Theorem~\ref{T1.1} that will
be an essential step in our analysis of the general case.

\begin{theorem}\label{T3.1} Adopt the conventions established above.
Let $(a,b)\in\mathcal{O}$, let $|(a,b)|\le R$, and let
$\operatorname{dist}((a,b),\mathcal{O}^c)\ge\delta$. Let
\begin{eqnarray*}
&&\theta_n(a):=\left\{\begin{array}{lll}
1&\text{if}&n=0\\
a(a+1)\dots(a+n-1)&\text{if}&n\ge1\end{array}\right\},\\
&&c_n(a,b):=\displaystyle\frac{(\theta_n(a)+\theta_n(b))2^{n-1}
\Gamma\left(\frac{1+n}2\right)}{\sqrt\pi(1-a-b-n)n!}\\
&&R_N(a,b):=\left|h_{a,b}^\IN(t)-\left\{
c(a,b)t^{(1-a-b)/2}+\sum_{n=0}^Nc_n(a,b)t^{n/2}\right\}\right|\,.
\end{eqnarray*}
\begin{enumerate}
\item If $\Re(a+b)>-7$, then we have a uniform estimate of the form
$$
R_N(a,b)\le\kappa(N,R)(1+\delta^{-4})\left\{t^{(N+1)/2}+e^{-1/(36t)}\right\}\,.
$$
\item If $\Re(a+b)\le-7$, then for each $(a,b)$ we have an estimate of the form
$$
R_N(a,b)\le \mathcal{E}_N(a,b)t^{(N+1)/2}\,.
$$
\end{enumerate}
\end{theorem}

We shall need the uniform estimate of Assertion~(1) for $0<\Re(a+b)<2$
 in examining the $\log$ terms
when $a+b=1$, $a<1$, and $a\in\mathbb{R}$ subsequently in Section~\ref{S3.8};
as this degree of precision is not needed for $\Re(a+b)\le0$, we have
contented ourselves  for the sake of brevity with a weaker estimate in Assertion~(2).

We note that the terms in $t^{n/2}$ when $n$ is odd can be regarded as arising
from a boundary contribution at $x=1$; the terms when $n$ is even can be regarded as arising
both from a boundary contribution
at $x=1$ and from regularized interior integrals. The term in $t^{(1-a-b)/2}$ can be
regarded as arising from a boundary contribution at $x=0$ and this is the only term which
arises in this way. Although a-priori $c(a,b)$ and $c_n(a,b)$
 could have poles when $a+b=0,-2,\dots$, we will show
in Lemma~\ref{L3.6} that they are
regular on $\mathcal{O}$.
We have, for example,
\begin{equation}\label{E3.b}
\begin{array}{l}
\displaystyle c_0(a,b)=\frac{(1+1)2^{-1}
\Gamma\left(\frac{1}2\right)}{\sqrt\pi(1-a-b)}=\frac1{1-a-b},\qquad
\displaystyle c_1(a,b)=\frac{(a+b)2^{0\vphantom{A^A}}
\Gamma(1)}{\sqrt\pi(-a-b)}=-\frac1{\sqrt\pi},\\
\displaystyle c_2(a,b)=-\frac{a^2+a+b^2+b}{2(1+a+b)}\,.
\end{array}\end{equation}

Theorem~\ref{T1.2} shows there are log terms when $a+b=1$ and gives a very precise
statement of what the terms are in the asymptotic expansion. In the following result, we
give a complete asymptotic expansion with log terms when $a+b=1-2k$
and when $(a,b)\in\mathbb{R}^2$. We shall identify
the coefficient of the $\log(t)t^k$ term. We do not, however, identify the remaining
coefficients explicitly although our methods would permit this.
We must restrict to $(a,b)\in\mathbb{R}^2$ as we shall use a monotonicity argument;
we do not know how to handle the case when $a+b=1-2k$ and $\Im(a)\ne0$.

\begin{theorem}\label{T3.2} Let $(a,b)\in\mathbb{R}^2$ satisfy $a<1$, $b<1$, and
$a+b=1-2k$
where $k$ is a non-negative integer. There exist functions $\tilde c_n(a,b)$
which are real analytic in $(a,b)$ so that if $N\ge 2k+1$, then:
$$
h_{a,b}^\IN(t)=\sum_{n=0}^N\tilde c_n(a,b)t^{n/2}
-\frac{a(a+1)\dots(a+2k-1)}{2\cdot k!}\log(t)t^k+O(t^{(N+1)/2})\,.$$
\end{theorem}

Here is a brief outline to Section~\ref{S3}. In Section~\ref{S3.1}, we use Theorem~\ref{T1.3} to
show that the heat content functions of $\IN$ in $S^1$ and or of $\IN$ in $\mathbb{R}$
are the same up to an exponentially small error in $t$ under very mild assumptions on $\phi$ and $\rho$
(see Lemma~\ref{L3.3}). We then give two recursion
relations that relate heat content functions for different values of $\{a,b\}$. These will
be used subsequently in various induction arguments. Lemma~\ref{L3.4} deals
quite generally with heat content functions on the interval where the ambient manifold
is $S^1$ and where $\phi$ and $\rho$ satisfy certain vanishing and smoothness conditions;
Lemma~\ref{L3.5} deals with the heat content on the interval when
$\phi(x)=x^{-a}$ and $\rho(x)=x^{-b}$
where the ambient manifold is $\mathbb{R}$.
In Section~\ref{S3.2}
(see Lemma~\ref{L3.6}), we show that the function $c(a,b)$
of Theorem~\ref{T1.1} and the functions $c_n(a,b)$ of Theorem~\ref{T3.1}
are holomorphic on $\mathcal{O}$ despite the apparent poles when $a+b=-2k$ is an
even non-positive integer. In Section~\ref{S3.3}
(see Lemma~\ref{L3.7}), we establish Theorem~\ref{T3.1} in the special case that
$b=0,-1,-2,\dots$. Since $\rho$ is smooth in this case, the asymptotic series follows
from Theorem~\ref{T1.4}. However, we must reprove this result since we need a very
precise control on the remainder terms.
On the half line $\HL$, one is tempted to define:
\begin{equation}\label{E3.c}
\begin{array}{l}
\displaystyle h_{a,b}^\HL(t):=\frac1{\sqrt{4\pi t}}
\int_0^\infty\int_0^\infty x^{-a}y^{-b}e^{-(x-y)^2/(4t)}dydx
\\
\qquad\qquad\hphantom{a}=\displaystyle\frac1{\sqrt{4\pi t}}
\int_0^{\infty\vphantom{\vrule height 6pt}}
\int_0^x\{x^{-a}y^{-b}+x^{-b}y^{-a}\}e^{-(x-y)^2/(4t)}dydx
\,.\end{array}\end{equation}
Although the integral of Equation~(\ref{E3.a}) converges for
$\Re(a)<1$ and $\Re(b)<1$, the integral
of Equation~(\ref{E3.c}) only converges if we assume additionally that
$\Re(a+b)>1$. Thus we must regularize this integral for other points $(a,b)$ of
$\mathcal{O}$ and
define a regularized heat content function.
In Section~\ref{S3.4} (see Lemma~\ref{L3.8}), we prove a technical result that we
will use to regularize the integral of Equation~(\ref{E3.c}). This result is Mathematica
\cite{M9} assisted and is the reason we assumed $-7<\Re(a+b)$ in Theorem~\ref{T3.1}.
In Section~\ref{S3.5} (see Lemma~\ref{L3.9}),
we study a regularized heat trace on the half line and in
Section~\ref{S3.6} (see Lemma~\ref{L3.10}), we use the regularized
heat trace on the half line of
Lemma~\ref{L3.9} to study a corresponding regularized heat trace on the interval.
The regularized heat trace of Lemma~\ref{L3.10} together with the expansion of
Lemma~\ref{L3.7} is then used in Section~\ref{S3.7}
to complete the proof of Theorem~\ref{T3.1};
the recursion relations of Section~\ref{S3.3} play a central role in this analysis.
In Section~\ref{S3.8},
we begin the proof of Theorem~\ref{T3.2} by examining the situation when $a+b=1$,
$a<1$, and $a\in\mathbb{R}$.
We complete
the proof of Theorem~\ref{T3.2} in Section~\ref{S3.9} by using Lemma~\ref{L3.4} once again.

\subsection{Recursion relations}\label{S3.1}
The following lemma follows immediately from Theorem~\ref{T1.3}.

\begin{lemma}\label{L3.3}
Let $D_{\mathbb{R}}=-\partial_x^2$ on $\mathbb{R}$
 and let $D_{S^1}=-\partial_{x^2}$ on $S^1$. Let $M=S^1$ or $M=\mathbb{R}$, and let
 $\tilde M=[-1,2]$.
If $\Re(a)\le0$ and $\Re(b)<1$, then
$$
|\beta_\Omega(\phi,\rho,D_{S^1})(t)-
\beta_\Omega^{\tilde M}(\phi,\rho,D_{\mathbb{R}})(t)|\le 2^{(2+m)/2}
\frac1{1-\Re(b)}e^{-1/(8t)}\,.
$$
\end{lemma}

 We first suppose the ambient manifold is the circle $S^1$.
A {\it spectral resolution} for $D_{S^1}$ and
the corresponding {\it Fourier coefficients} for $\phi\in L^1(S^1)$ are given
by
$$
\left\{\frac{e^{\sqrt{-1}n}}{\sqrt{2\pi}},n^2\right\}_{n=-\infty}^\infty\text{ and }
\gamma_n(\phi):=\frac1{\sqrt{2\pi}}\int_{-\pi}^\pi\phi(x)e^{-\sqrt{-1}nx}dx\,.
$$
If $\phi$ and $\rho$ are in $L^1$, then the Fourier coefficients are uniformly
bounded and we have for $t>0$ that:
\begin{eqnarray*}
&&e^{-tD_{S^1}}\phi=\sum_{n=-\infty}^\infty e^{-tn^2}\gamma_n(\phi)e^{\sqrt{-1}nx}\text{ and }
\beta(\phi,\rho,D_{S^1})(t)
=\sum_{n=-\infty}^\infty e^{-tn^2}\gamma_n(\phi)\gamma_{-n}(\rho)\,.
\end{eqnarray*}
Let $C_0[0,1]$ be the space of continuous functions on the interval which vanish
at $x=0$
and at $x=1$. Let $C_0^1[0,1]$ be the space of continuously differentiable functions $\phi$ on
the interval so $\phi$ and $\phi^\prime$ both vanish at $x=0$ and at $x=1$.
\begin{lemma}\label{L3.4}
\
\begin{enumerate}
\item If $\phi\in C^\infty(0,1)\cap C_0[0,1]$, and
if $\phi^\prime\in L^1[0,1]$, then
$\gamma_n(\phi^\prime)=\sqrt{-1}n\gamma_n(\phi)$.
\smallbreak\item
If $\phi\in C^\infty(0,1)\cap C^1_0[0,1]$,
if $\phi^{\prime\prime}\in L^1[0,1]$, and if $\rho\in L^1[0,1]$, then:
\begin{enumerate}
\item
$\beta_\IN(\phi^{\prime\prime},\rho,D_{S^1})(t)=\partial_t\beta_\IN(\phi,\rho,D_{S^1})(t)$.
\item $\beta_\IN(\phi,\rho,D_{S^1})(t)$ is continuous.
\item  $|\beta_\IN(\phi,\rho,D_{S^1})(t)|\le\frac23\pi^2
||\phi^{\prime\prime}||_{L^1(S^1)}||\rho||_{L^1(S^1)}$ for $t\in\HL$.
\end{enumerate}
\smallbreak\item
If $\phi,\rho\in C^\infty(0,1)\cap C_0[0,1]$,
if $\phi^\prime\in L^1[0,1]$, and if $\rho^\prime\in L^1[0,1]$, then:
\begin{enumerate}
\item $\beta_\IN(\phi^\prime,\rho^\prime,D_{S^1})(t)
=-\partial_t\beta_\IN(\phi^\prime,\rho^\prime,D_{S^1})(t)$.
\item $\beta_\IN(\phi,\rho,D_{S^1})(t)$ is continuous.
\item
 $|\beta_\IN(\phi,\rho,D_{S^1})(t)|\le\frac23\pi^2
 {\pi^2|\phi^\prime|_{L^1(S^1)}||\rho^\prime||_{L^1(S^1)}}$ for $t\in\HL$.\end{enumerate}
\end{enumerate}\end{lemma}

\begin{proof} Suppose $\phi\in C^\infty(0,1)\cap C_0[0,1]$ and that $\phi^\prime\in L^1[0,1]$.
We establish Assertion~(1) by computing:
\begin{eqnarray*}
&&\lim_{\epsilon\downarrow0}\int_\epsilon^{1-\epsilon}
\partial_x\left\{\phi e^{-\sqrt{-1}nx}dx\right\}dx=\lim_{\epsilon\downarrow0}
\left.\left\{\phi(x)e^{-\sqrt{-1}nx}\right\}\right|_{x=\epsilon}^{1-\epsilon}=
\left.\left\{\phi(x)e^{-\sqrt{-1}nx}\right\}\right|_{x=0}^{1}=0\\
&=&\lim_{\epsilon\downarrow0}\left\{\int_\epsilon^{1-\epsilon}\phi^\prime(x) e^{-\sqrt{-1}nx}
-\sqrt{-1}n\phi(x)e^{-\sqrt{-1}nx}\right\}dx=\gamma_n(\phi^\prime)-n\sqrt{-1}\gamma_n(\phi)\,.
\end{eqnarray*}
Let $\phi\in C^\infty(0,1)\cap C_0^1[0,1]$, let $\phi^\prime\in L^1[0,1]$, and let
$\rho\in L^1[0,1]$. We can use Assertion (1) to see that
$\gamma_n(\phi^{\prime\prime})=-n^2\gamma_n(\phi)$ and hence
\begin{eqnarray*}
\beta_{\IN}(\phi^{\prime\prime},\rho,D_{S^1})(t)
&=&-\sum_nn^2e^{-tn^2}\gamma_n(\phi)\gamma_n(\rho)
=\partial_t\sum_ne^{-tn^2}\gamma_n(\phi)\gamma_{-n}(\rho)\\
&=&\displaystyle\partial_t\beta_\IN(\phi,\rho,D_{S^1})(t)\,.
\end{eqnarray*}
If $n\ne0$, then $|\gamma_n(\phi)|\le n^{-2}||\phi^{\prime\prime}||_{L^1(S^1)}$ and
$|\gamma_n(\rho)|\le||\rho||_{L^1(S^1)}$. This shows that the series defining
$\beta_\IN(\phi,\rho,D_{S^1})(t)$ converges uniformly
on $\HL$ and hence is continuous on $\HL$. We estimate:
$$|\beta_\IN(\phi,\rho,D_{S^1})(t)\le||\phi^{\prime\prime}||_{L^1(S^1)}||\rho||_{L^1(S^1)}
\sum_{n\ne0}\frac1{n^2}=||\phi^{\prime\prime}||_{L^1(S^1)}||\rho||_{L^1(S^1)}
\frac23\pi^2\,.
$$
Similarly, if $\phi\in C^\infty(0,1)\cap C_0[0,1]$, if $\rho\in C^\infty(0,1)\cap C_0[0,1]$,
if $\phi^\prime\in L^1[0,1]$, and if $\rho^\prime\in L^1[0,1]$, then:
\begin{eqnarray*}
&&\gamma_n(\phi^\prime)=\sqrt{-1}n\gamma_n(\phi),\quad
\gamma_{-n}(\rho^\prime)=-\sqrt{-1}n\gamma_{-n}(\rho),\\
&&\beta_{\IN}(\phi^\prime,\rho^\prime,D_{S^1})(t)
=+\sum_nn^2e^{-tn^2}\gamma_n(\phi)\gamma_n(\rho)
=-\partial_t\sum_ne^{-tn^2}\gamma_n(\phi)\gamma_n(\rho)\\
&&\qquad\qquad=-\partial_t\beta_\IN(\phi,\rho,D_{S^1})(t)\,.
\end{eqnarray*}
We have $\gamma_0(\phi^{\prime\prime})=0$.
If $n\ne0$, we may estimate $|\gamma_n(\rho)|\le\frac1n|\phi^\prime|_{L^1}$
and $|\gamma_n(\rho)|\le\frac1n|\rho^\prime|_{L^1}$. Thus the series defining
$\beta_\IN(\phi,\rho,D_{S^1})(t)$ converges uniformly
on $\HL$ and hence is continuous in $\HL$ and we have a similar estimate.
\end{proof}

Next we assume that the ambient manifold is $\mathbb{R}$,
that $\phi(x)=x^{-a}$, and that $\rho(x)=x^{-b}$.
Let $\theta_n(b)$ be as in Theorem~\ref{T3.1}.
Let $|(a,b)|=\{|a|^2+|b|^2\}^{1/2}$.

\begin{lemma}\label{L3.5} Let $\Re(a)<1-\delta$, let $\Re(b)<1-\delta$, and let $|(a,b)|<R$. Then:
\begin{eqnarray*}
&&\left|h_{a,b}^\IN(t)-\frac1{2t(1-a)}
\left\{h_{a-2,b}^\IN(t)-h_{a-1,b-1}^\IN(t)
-\frac1{\sqrt\pi}\sum_{n=0}^Nt^{(n+1)/2}2^{n}\frac{\theta_n(b)}{n!}\Gamma\left(\frac{n+2}2\right)
\right\}\right|\\
&&\qquad\le\mathcal{E}_N(\delta,R,t)\,.
\end{eqnarray*}
\end{lemma}

\begin{proof}
We integrate by parts in $x$ with
$u=\tilde x^{-b}e^{-(x-\tilde x)^2/(4t)}$, $v=x^{1-a}/(1-a)$, and
$dv=x^{-a}dx$. Since $(uv)(0)=0$, we have that:
\begin{eqnarray*}
&&h_{a,b}^\IN(t)=\frac1{\sqrt{4\pi t}}\int_0^1\int_0^1
x^{-a}\tilde x^{-b}e^{-(x-\tilde x)^2/(4t)}dxd\tilde x\\
&=&\frac1{2t\sqrt{4\pi t}(1-a)}\left\{\int_0^1\int_0^1x^{1-a}\tilde x^{-b}(x-\tilde x)e^{-(x-\tilde x)^2/(4t)}
dxd\tilde x
-\int_0^1\tilde x^{-b}(1-\tilde x)e^{-(1-\tilde x)^2/(4t)}d\tilde x\right\}\\
&=&\frac1{2t(1-a)}
\left\{h_{a-2,b}^\IN(t)-h_{a-1,b-1}^\IN(t)+\mathcal{Z}_{a,b}(t)\right\}\,.
\end{eqnarray*}
We replace $\tilde x$
by $1-\tilde x$ and set $u=\tilde x/\sqrt{4t}$ to see:
$$
\mathcal{Z}_{a,b}(t)=
-\frac1{\sqrt{4\pi t}}\int_0^1(1-\tilde x)^{-b}\tilde xe^{-\tilde x^2/(4t)}d\tilde x
=-\frac1{\sqrt\pi}(4t)^{1/2}\int_0^{1/\sqrt{4t}}(1-\sqrt{4t}u)^{-b}ue^{-u^2}du\,.
$$
Let $\theta_n$ be as in Theorem~\ref{T3.1}.
We expand the function $(1-\eta)^{-b}$ in a Taylor series about $\eta=0$:
\begin{equation}\label{E3.d}
(1-\eta)^{-b}=\sum_{n=0}^N\frac{\theta_n(b)}{n!}\eta^n+R_{N,b}(\eta)
\end{equation}
 to see
$$
\mathcal{Z}_{a,b}(t)=-\frac1{\sqrt\pi}\sum_{n=0}^N(4t)^{(n+1)/2}
\int_0^{1/\sqrt{4t}}\frac{\theta_n(b)}{n!}u^{n+1}e^{-u^2}du+O(t^N)\,.
$$
We may now replace the upper integral by $\infty$ modulo an exponentially
small error and evaluate the coefficients to complete the proof by expanding
\medbreak\hfill
$\displaystyle\mathcal{Z}_{a,b}(t)=-\frac1{\sqrt\pi}\sum_{n=0}^Nt^{(n+1)/2}2^{n}\frac{\theta_n(b)}{n!}\Gamma\left(\frac{n+2}2\right)$.\hfill\vphantom{.}
\end{proof}

\subsection{Regularity of $c(a,b)$ and $c_n(a,b)$ when $a+b=0,-2,-4,\dots$}\label{S3.2}
\begin{lemma}\label{L3.6}
The functions $c(a,b)$ and
$c_n(a,b)$ of Equation~(\ref{E3.b}) are holomorphic on $\mathcal{O}$.
\end{lemma}

\begin{proof} Only points $(a,b)$ of $\mathcal{O}$
 where $a+b=-2k$ is a non-positive even integer are at issue.
Since $\Re(a)+\Re(b)<2$, $\Gamma\left(\frac{2-a-b}2\right)$ is always
regular. Since $\Re(a)<1$, $\Re(b)<1$,  $\Gamma(1-a)/\Gamma(b)$ and
$\Gamma(1-b)/\Gamma(a)$ are regular as well. Since
$\Gamma(a+b-1)$ has a simple pole when $a+b=1,0,-1,-2,\dots$,
we must show that
$$\frac{\Gamma(1-a)}{\Gamma(b)}+\frac{\Gamma(1-b)}{\Gamma(a)}$$
vanishes when $a+b=-2k$ is a non-positive even integer.
The {\it Euler reflection formula} (see, for example,
Formula 8.334.3 in \cite{grad65}) implies:
$$
\Gamma(\epsilon-2k)\Gamma(2k+1-\epsilon)=
-\Gamma(1+\epsilon)\Gamma(-\epsilon)\,.
$$
Assume $a+b=-2k$ is a non-positive even integer.
We set $\epsilon=-b$ to see:
\begin{eqnarray*}
&&\Gamma(a)\Gamma(1-a)+\Gamma(b)\Gamma(1-b)
=\Gamma(-2k-b)\Gamma(1+2k+b)+\Gamma(b)\Gamma(1-b)\\
&=&\Gamma(-2k+\epsilon)\Gamma(1+2k-\epsilon)
+\Gamma(-\epsilon)\Gamma(1+\epsilon)
=\{1-1\}\Gamma(1+\epsilon)\Gamma(-\epsilon)=0\,.
\end{eqnarray*}

The denominator $(1-a-b-n)^{-1}$ defining $c_n$ has a simple pole when
$a+b=-n+1$.  Thus to show $c_n(a,b)$ is regular
on $\mathcal{O}$, we must show $\theta_n(a)+\theta_n(b)=0$
when $a+b=1-n$ where $n$ is an odd positive integer.
Set $b=1-n-a$. We have
\begin{eqnarray*}
\theta_n(a)+\theta_n(1-n-a)&=&\left\{a(a+1)\dots(a+(n-1))\right\}+\left\{(1-n-a)\dots(-a)\right\}\\
&=&\{1+(-1)^n\}a(a+1)\dots(a+(n-1))\,.
\end{eqnarray*}
This vanishes if $n$ is odd.
\end{proof}

\subsection{Heat content asymptotics if $m=1$,
if $\phi(x)=x^{-a}$, and if $\rho(x)=x^{-b}$ is smooth}\label{S3.3}
The following special case will play a central role in our regularization procedure. As we need
a very precise control on the remainder estimate, we can not simply appeal to Theorem~\ref{T1.4}.

\begin{lemma}\label{L3.7}
Let $R>0$, let $\delta>0$, let $b$ be a non-positive integer, and let $N$ be a non-negative integer
with $N\ge\Re(a+b)+3$..
 There exist holomorphic functions $\{C_b(a),C_{b,n}(a)\}$ defined for
 $\Re(a)<1$ and $n=0,1,\dots$ so that if $\Re(a)<1-\delta$ and if $|a|\le R$, then:
\begin{eqnarray*}
&&\left|h_{a,b}^\IN(t)-\left\{C_b(a)t^{(1-a-b)/2}
+\sum_{n=0}^NC_{b,n}(a)t^{n/2}\right\}\right|\\
&\le&\kappa(b,N,R)(1+\delta^{-1})
\left\{t^{(N+1)/2}
+t^{(1-b-\Re(a))/2}e^{-1/(36t)}\right\}\,.
\end{eqnarray*}
\end{lemma}

\begin{proof} Let $\Xi\in C^\infty(\mathbb{R})$ be a smooth monotonically
decreasing cut-off function satisfying
\begin{equation}\label{E3.e}
\textstyle
\Xi(x)=1\text{ for }x\le\frac13\quad\text{and}\quad
\Xi(x)=0\text{ for }x\ge\frac23\,.
\end{equation}
The choice of $\Xi$ will, of course, play no role in the final
formulas. Express
\begin{eqnarray*}
h_{a,b}^\IN(t)&=&\frac1{\sqrt{4\pi t}}\int_0^1\int_0^1
\Xi(x)x^{-b}\tilde x^{-a}e^{-(x-\tilde x)^2/(4t)}dxd\tilde x\\
&+&\frac1{\sqrt{4\pi t}}\int_0^1\int_0^1
(1-\Xi(x))x^{-b}\tilde x^{-a}e^{-(x-\tilde x)^2/(4t)}dxd\tilde x\,.
\end{eqnarray*}
We replace $x$ by $1-x$ and $\tilde x$ by $1-\tilde x$ to see:
$$
h_{a,b}^\IN(t)=
\frac1{\sqrt{4\pi t}}\int_0^1\int_0^1\left\{\Xi(x)x^{-b}\tilde x^{-a}
+(1-\Xi(1-x))(1-x)^{-b}(1-\tilde x)^{-a}\right\}
e^{-(x-\tilde x)^2/(4t)}dxd\tilde x\,.
$$
Since $\Xi(x)$ and $(1-\Xi(1-x))$ vanish for $x\ge1$, we may let the $dx$ integral range from $1$ to
$\infty$:
$$
h_{a,b}^\IN(t)=
\frac1{\sqrt{4\pi t}}\int_0^1\int_0^\infty\left\{\Xi(x)x^{-b}\tilde x^{-a}+(1-\Xi(1-x))(1-x)^{-b}(1-\tilde x)^{-a}\right\}
e^{-(x-\tilde x)^2/(4t)}dxd\tilde x\,.
$$
Recall that $b$ is a non-positive integer, and decompose:
\medbreak\qquad
$\displaystyle h_{a,b}^\IN(t)=\mathcal{A}_{a,b}(t)+
\mathcal{B}_{a,b}(t)+\mathcal{C}_{a,b}(t)$ where
\medbreak\qquad
$\displaystyle\mathcal{A}_{a,b}(t):=
\frac1{\sqrt{4\pi t}}\int_0^1\int_{-\infty}^\infty
\Xi(x)x^{-b}\tilde x^{-a}e^{-(x-\tilde x)^2/(4t)}dxd\tilde x$
\medbreak\qquad\qquad\qquad$\displaystyle
+\frac1{\sqrt{4\pi t}}\int_0^1\int_{-\infty}^\infty
(1-\Xi(1-x))(1-x)^{-b}(1-\tilde x)^{-a}e^{-(x-\tilde x)^2/(4t)}dxd\tilde x$,
\medbreak\qquad
$\displaystyle\mathcal{B}_{a,b}(t):=-
\frac1{\sqrt{4\pi t}}\int_0^1\int_{-\infty}^0
\Xi(x)x^{-b}\tilde x^{-a}e^{-(x-\tilde x)^2/(4t)}dxd\tilde x$,
\medbreak\qquad
$\displaystyle\mathcal{C}_{a,b}(t):=-
\frac1{\sqrt{4\pi t}}\int_0^1\int_{-\infty}^0
(1-\Xi(1-x))(1-x)^{-b}(1-\tilde x)^{-a}e^{-(x-\tilde x)^2/(4t)}dxd\tilde x$.
\medbreak\noindent
The two integrals defining $\mathcal{A}_{a,b}(t)$ for $x\le0$ and the
integrals defining $\mathcal{B}_{a,b}(t)$ and $\mathcal{C}_{a,b}(t)$
converge absolutely at $x=-\infty$ owing to the exponential damping
$e^{-(x-\tilde x)^2/(4t)}\le e^{-(x^2+\tilde x^2)/(4t)}$ for $x\le 0$.
We again replace $x$ by $1-x$ and $\tilde x$ by $1-\tilde x$
in the second integral defining $\mathcal{A}_{a,b}(t)$ and use the fact that
$\Xi(x)+(1-\Xi(x))=1$ to express:
\begin{equation}\label{E3.f}
\mathcal{A}_{a,b}(t)=
\frac1{\sqrt{4\pi t}}\int_0^1\int_{-\infty}^\infty x^{-b}\tilde x^{-a}e^{-(x-\tilde x)^2/(4t)}dxd\tilde x\,.
\end{equation}
We replace $x$ by $-x$. Since $\Xi(x)=1$ for $x\le0$
and $\Xi(1-x)=0$ for $x\le0$, we have:
\begin{eqnarray}
&&\mathcal{B}_{a,b}(t)=(-1)^{b+1}
\frac1{\sqrt{4\pi t}}\int_0^1\int_0^\infty
x^{-b}\tilde x^{-a}e^{-(x+\tilde x)^2/(4t)}dxd\tilde x,\label{E3.g}\\
&&\mathcal{C}_{a,b}(t)=-
\frac1{\sqrt{4\pi t}}\int_0^1\int_0^\infty
(1+x)^{-b}(1-\tilde x)^{-a}e^{-(x+\tilde x)^2/(4t)}dxd\tilde x\,.\label{E3.h}
\end{eqnarray}

We complete the proof by expanding each of these integrals separately.
We change variables and set $u=(x-\tilde x)/\sqrt{4t}$ in
Equation~(\ref{E3.f}) to express:
\begin{eqnarray*}
\mathcal{A}_{a,b}(t)&=&\frac1{\sqrt\pi }\int_0^1\int_{-\infty}^\infty
(\tilde x+\sqrt{4t}u)^{-b}\tilde x^{-a}e^{-u^2}dud\tilde x\,.
\end{eqnarray*}
Since $-b$ is a non-negative integer, we may
 expand $(\tilde x+\sqrt{4t}u)^{-b}$ using the binomial theorem:
$$
(\tilde x+\sqrt{4t}u)^{-b}=\sum_{n=0}^{-b}\frac{(-b)!}{(-b-n)!n!}
\tilde x^{-b-n}(4t)^{n/2}u^n\,.
$$
This shows that
$$
\mathcal{A}_{a,b}(t)=\frac1{\sqrt\pi }\sum_{n=0}^{-b}(4t)^{n/2}
\frac{(-b)!}{(-b-n)!n!}
\int_0^1\tilde x^{-a-b-n}d\tilde x\cdot\int_{-\infty}^\infty u^ne^{-u^2}du\,.
$$
Since $-b-n\ge0$ and $\Re(a)<1$, the integrals converge absolutely to define
holomorphic functions of $a$.
This gives rise to an expansion of the form given in Lemma~\ref{L3.7}.
The remainder term is identically equal to $0$.

Next, we examine $\mathcal{B}_{a,b}(t)$.
We interchange the order of integration and replace $x$ by $\sqrt{4t}x$
and $\tilde x$ by $\sqrt{4t}\tilde x$ in Equation~(\ref{E3.g}) to express:
$$\mathcal{B}_{a,b}(t):=t^{(1-a-b)/2}(-1)^{b+1}\frac{2^{(1-a-b)}}{\sqrt\pi}
\int_0^{1/\sqrt{4t}}\int_0^\infty
x^{-b}\tilde x^{-a}e^{-(x+\tilde x)^2}dxd\tilde x\,.
$$
Let
$$ C_b(a):=(-1)^{b+1}\frac{2^{(1-a-b)}}{\sqrt\pi}
\int_0^\infty\int_0^\infty
x^{-b}\tilde x^{-a}e^{-(x+\tilde x)^2}dxd\tilde x\,.$$
We decompose
$\mathcal{B}_{a,b}(t)=\mathcal{B}_{a,b}^1(t)-\mathcal{B}_{a,b}^2(t)$ where
\begin{eqnarray*}
&&\mathcal{B}_{a,b}^1(t):=t^{(1-a-b)/2}(-1)^{b+1}\frac{2^{(1-a-b)}}{\sqrt\pi}
\int_0^\infty\int_0^\infty
x^{-b}\tilde x^{-a}e^{-(x+\tilde x)^2}dxd\tilde x\\
&&\hphantom{\mathcal{B}_{a,b}^1(t):}=t^{(1-a-b)/2}C_b(a),\\
&&\mathcal{B}_{a,b}^2(t):=t^{(1-a-b)/2}(-1)^{b+1}\frac{2^{(1-a-b)}}{\sqrt\pi}
\int_{1/\sqrt{4t}}^\infty\int_0^\infty
x^{-b}\tilde x^{-a}e^{-(x+\tilde x)^2}dxd\tilde x\,.
\end{eqnarray*}
If $x\ge\frac1{\sqrt{4t}}$, then
$e^{-(x+\tilde x)^2}\le e^{-1/(8t)}e^{-(x^2+\tilde x^2)/2}$.
This permits us to estimate:
\begin{eqnarray*}
\left|\mathcal{B}_{a,b}^2(t)\right|
&\le& t^{(1-\Re(a+b))/2}e^{-1/(8t)}\frac{2^{1-\Re(a+b)}}{\sqrt\pi}
\int_0^\infty x^{-b}e^{-x^2/2}dx
\cdot\int_0^\infty \tilde x^{-\Re(a)}e^{-\tilde x^2/2}d\tilde x\,.
\end{eqnarray*}
The constant $|\pi^{-1/2}2^{(1-a-b)}|$ can be uniformly bounded since $|a|\le R$.
The $dx$ integral is convergent since $-b\ge0$. Since $|a|\le R$, the
 $d\tilde x$ integral is convergent on $\NHL$
and can be estimated uniformly by some suitably chosen constant $\kappa(R)$.
The $d\tilde x$ integral also is convergent on $\IN$ since
$\Re(a)\le1-\delta$ implies $-\Re(a)\ge\delta-1$ and consequently
can be estimated uniformly by
$\delta^{-1}$. This gives rise
to a remainder term of the form given.

Finally, we replace $x$ and $\tilde x$ by $\sqrt{4t}x$ and $\sqrt{4t}\tilde x$ in Equation~(\ref{E3.h})
to express
$$\mathcal{C}_{a,b}(t):=-
\frac{\sqrt{4t}}{\sqrt{\pi}}\int_0^{1/\sqrt{4t}}\int_0^\infty
(1+\sqrt{4t}x)^{-b}(1-\sqrt{4t}\tilde x)^{-a}e^{-(x+\tilde x)^2}dxd\tilde x\,.
$$
We use the binomial theorem to expand $(1+\sqrt{4t}x)^{-b}$ in
powers of $\sqrt{4t}x$. Thus, after suppressing the normalizing constants (which
can be uniformly bounded),
it suffices to consider an integral of the form
$$\mathcal{C}_{a,b,k}(t)=t^{(1+k)/2}\int_0^{1/\sqrt{4t}}\int_0^\infty x^k(1-\sqrt{4t}\tilde x)^{-a}
e^{-(x+\tilde x)^2}dxd\tilde x\text{ for }0\le k\le -b\,.$$
We decompose $\mathcal{C}_{a,b,k}(t)=\mathcal{C}_{a,b,k}^1(t)+\mathcal{C}_{a,b,k}^2(t)$ where
\begin{eqnarray*}
&&\mathcal{C}_{a,b,k}^1(t)=t^{(1+k)/2}\int_0^{1/(3\sqrt{4t})}
\int_0^\infty x^k(1-\sqrt{4t}\tilde x)^{-a}
e^{-(x+\tilde x)^2}dxd\tilde x,\\
&&\mathcal{C}_{a,b,k}^2(t)
=t^{(1+k)/2}\int_{1/(3\sqrt{4t})}^{1/\sqrt{4t}}\int_0^\infty x^k(1-\sqrt{4t}\tilde x)^{-a}
e^{-(x+\tilde x)^2}dxd\tilde x\,.
\end{eqnarray*}
As before, we see that the integral defining $\mathcal{C}_{a,b,k}^2$ is
exponentially damped and satisfies an estimate
$$
|\mathcal{C}_{a,b,k}^2|\le \kappa(b,R)(1+\delta^{-1})e^{-1/(36t)}\,.
$$
Expand $(1-w)^{-a}$ in a Taylor series for $|w|\le\frac23$ to see:
$$
\left|(1-w)^{-a}-\sum_{n=0}^N\frac{\theta_n(a)}{n!}w^n\right|
\le\kappa(N,R)|w|^{N+1}
$$
for suitably defined holomorphic functions $\theta_n(a)$.
Replacing $w$ by $\sqrt{4t}\tilde x$ and noting $|\sqrt{4t}\tilde x|\le\frac13$ for
$\tilde x\in[0,1/(3\sqrt{4t})]$, then
leads to an expansion of the desired form.
\end{proof}

\subsection{Regularizing the integrals on $\mathbb{R}$}\label{S3.4}
Let $\mathcal{U}:=\{(a,b)\in\mathbb{C}^2:\Re(a)<1\text{ and }\Re(b)<1\}$.
Let
\begin{eqnarray*}
&&\mathcal{O}_{-1}:=\{(a,b)\in\mathcal{U}:1<\Re(a+b)<2\},\\
&&\mathcal{O}_0:=\{(a,b)\in\mathcal{U}:-1<\Re(a+b)<1\},\\
&&\mathcal{O}_1:=\{(a,b)\in\mathcal{U}:-3<\Re(a+b)<0\text{ and }
a+b\ne-1\},\\
&&\mathcal{O}_2:=\{(a,b)\in\mathcal{U}:-5<\Re(a+b)<-1\text{ and }
a+b\ne-1,-2,-3\}\\
&&\mathcal{O}_3:=\{(a,b)\in\mathcal{U}:-7<\Re(a+b)<-2\text{ and }
a+b\ne-1,-2,-3,-4,-5\}\,.
\end{eqnarray*}
We have $0\in\mathcal{O}_0$, $-2\in\mathcal{O}_1$,
$-4\in\mathcal{O}_2$, $-6\in\mathcal{O}_3$.
Furthermore
$$
\mathcal{O}_k\cap\mathcal{O}_{k+1}\ne\emptyset\text{ for }k\ge0
\text{ and }
\bigcup_{k=0}^3\mathcal{O}_k=\{(a,b)\in\mathcal{O}:-7<\Re(a+b)<1\}\,.
$$

\begin{lemma}\label{L3.8}
Let $-1\le k\le 3$, let $R>0$, and let $0<\delta<\frac12$.
There exist holomorphic functions $\sigma_{k,\ell}$ which are defined
on $\mathcal{O}_k$ so that if we define
$$
G_k(\eta;a,b):= \eta^{-a}+\eta^{-b}-\sum_{\ell=0}^k\sigma_{k,\ell}(a,b)
\{\eta^\ell+\eta^{-a-b-\ell}\}
$$
 then if $(a,b)\in\mathcal{O}_k$ with $|(a,b)|\le R$ and
 $\operatorname{dist}((a,b),\mathcal{O}_k^c)\ge\delta$, then:
\begin{enumerate}
\item
$\left|G_k(\eta;a,b)\right||\le\kappa(k,R)(1+\delta^{-k})(1-|\eta|)^{2k+2}$ for $\eta\in[\frac12,1]$.
\item $\left|G_k(\eta;a,b)\right|\le\kappa(k,R)(1+\delta^{-k})\eta^{\delta-1}$ for $\eta\in[0,\textstyle\frac12]$.
\end{enumerate}\end{lemma}

\begin{proof}Suppose first that $k=-1$ so the sum over $\ell$ does not
appear. Clearly $|\eta^{-a}+\eta^{-b}|$ is
bounded on $[\frac12,1]$ so Assertion~(1) holds.
Since $\Re(a)\le 1-\delta$ and $\Re(b)\le1-\delta$
we have $|\eta^{-a}+\eta^{-b}|\le2\eta^{1-\delta}$ on $[0,\frac12]$
and Assertion~(2) holds. The Lemma
now follows in this special case.

We introduce some additional notation to handle the other values of $k$.
For $i$ a non-negative integer, set:
\medbreak\qquad
$A_i(\eta):=\eta^{-a-b-i}+\eta^i$,\qquad\qquad
$\displaystyle B_i(\eta):=\frac{2(A_i(\eta)-A_{i+1}(\eta))}{(-2i-1-a-b)(1-\eta)^{2}}$,
\medbreak\qquad
$\displaystyle C_i(\eta):=\frac{6(B_i(\eta)-B_{i+1}(\eta))}{(-2i-2-a-b)(1-\eta)^{2}}$,\qquad\qquad
$\displaystyle D_i(\eta):=\frac{10(C_i(\eta)-C_{i+1}(\eta))}{(-2i-3-a-b)(1-\eta)^{2}}$,\medbreak\qquad
$\displaystyle E_i(\eta):=\frac{14(D_i(\eta)-D_{i+1}(\eta))}{(-2i-4-a-b)(1-\eta)^{2}}$,\qquad\qquad
$\displaystyle F_i(\eta):=\frac{18(E_i(\eta)-E_{i+1}(\eta))}{(-2i-5-a-b)(1-\eta)^{2}}$,\medbreak\qquad
$\displaystyle G_i(\eta):=\frac{22(F_i(\eta)-F_{i+1}(\eta))}{(-2i-6-a-b)(1-\eta)^{-2}}$,\qquad\quad
$\displaystyle H_i(\eta):=\frac{26(G_i(\eta)-G_{i+1}(\eta))}{(-2i-7-a-b)(1-\eta)^{-2}}$,\medbreak\qquad
$\displaystyle X_0(\eta):=\frac{A_0(\eta)-\eta^{-a}-\eta^{-b}}{ab(\eta-1)^2}$,\qquad\qquad\qquad\quad
$\displaystyle X_1(\eta):=\frac{12(X_0(\eta)-\frac12B_0(\eta))}{(-1-a-b-ab)(\eta-1)^2}$,\medbreak\qquad
$\displaystyle X_2(\eta):=\frac{30(X_1(\eta)-\frac12C_0(\eta))}{(-4-2a-2b-ab)(\eta-1)^2}$,\qquad\quad
$\displaystyle X_3(\eta):=\frac{56(X_2(\eta)-\frac12D_0(\eta))}{(-9-3a-3b-ab)(\eta-1)^2}$,\medbreak\qquad
$\displaystyle X_4(\eta)=\frac{90(X_3(\eta)-\frac12E_0(\eta))}{(-16-4a-4b-ab)(\eta-1)^2}$,\qquad\quad
$\displaystyle X_5(\eta)=\frac{132(X_4(\eta)-\frac12F_0(\eta))}{(-25-5a-5b-ab)(\eta-1)^2}$,\medbreak\qquad
$\displaystyle X_6(\eta)=\frac{182(X_5(\eta)-\frac12G_0(\eta))}{(-36-6a-6b-ab)(\eta-1)^2}$,\qquad\quad
$\displaystyle X_7(\eta)=\frac{240(X_6(\eta)-\frac12H_0(\eta))}{(-49-7a-7b-ab)(\eta-1)^2}$.
\medbreak\noindent
We used Mathematica \cite{M9} to see that all these functions are
regular at $\eta=1$ and involve coefficients which are polynomial in $\{a,b\}$; the divisions defining
these expansions are exact. Suppose first $k=0$. Then we may expand:
\begin{equation}\label{E3.i}
\eta^{-a}-\eta^{-b}-A_0(\eta)=-ab(\eta-1)^2X_0(\eta)\,.\end{equation}
This gives an error of $\kappa_0(R)|1-\eta|^2$. There are no denominators in the regularization; the coefficients are regular for all $(a,b)\in\mathbb{C}^2$
and, in particular, are holomorphic on $\mathcal{O}_0$.
Next suppose $k=1$. We have $\operatorname{Span}\{A_0,B_0\}=\operatorname{Span}\{A_0,A_1\}$.
The denominator $\{-1-a-b\}$ appears in $B_0$. We have explicitly avoided $a+b=-1$
so this is holomorphic on $\mathcal{O}_1$ and the denominator can be bounded
by $\delta^{-1}$. We combine the expansion in Equation~(\ref{E3.i})
with the expansion
\begin{equation}\label{E3.j}
X_0(\eta)=\frac12B_0(\eta)+\frac{-1-a-b-ab}{12}(\eta-1)^2X_1(\eta)
\end{equation}
to obtain an expansion of $\eta^{-a}+\eta^{-b}$ in $\{A_0,A_1\}$ with an error of
$\kappa_1(R)(1+\delta^{-1})|\eta-1|^4$. Next,
we suppose $k=2$. We have
$$\operatorname{Span}\{A_0,B_0,C_0\}=\operatorname{Span}\{A_0,B_0,B_1\}
=\operatorname{Span}\{A_0,A_1,A_2\}\,.$$
The denominators which appear are $(-1-a-b)^{-1}$ in $B_0$, $(-3-a-b)^{-1}$ in $B_1$, and
$(-2-a-b)^{-1}$ in $C_0$. We have excluded these values from $\mathcal{O}_2$ and the
denominators can be bounded by $\delta^{-1}$ on $\mathcal{O}_2$. We combine the expansion
in Equation~(\ref{E3.j}) with the expansion
\begin{equation}\label{E3.k}
X_1(\eta)=\frac12C_0+\frac{(-4-2a-2b-ab)}{30}(\eta-1)^2 X_2(\eta)
\end{equation}\
to obtain an expansion of $\eta^{-a}+\eta^{-b}$ in $\{A_0,A_1,A_2\}$ with
an error of $\kappa_2(R)(1+\delta^{-1})|\eta-1|^6$.
Finally, suppose $k=3$. We use $\{A_0,B_0,C_0,D_0\}$ or equivalently $\{A_0,A_1,A_2,A_3\}$.
We have the denominators $\{(-1-a-b)^{-1}$, $(-3-a-b)^{-1}$ ,$(-5-a-b)^{-1}\}$ in $\{B_0,B_1,B_2\}$,
$\{(-2-a-b)^{-1}$, $(-4-a-b)^{-1}\}$ in $\{C_0,C_1\}$, and
$(-3-a-b)^{-1}$ in $D_1$. We have excluded
these values from $\mathcal{O}_3$ and the coefficients can be bounded by $\delta^{-2}$ since
the denominator $(-3-a-b)^{-1}$ appears twice. We combine the expansion in
Equation~(\ref{E3.k}) with the expansion
$$X_2(\eta)=\frac12D_0+\frac{(-9-3a-3b-ab)}{56}(\eta-1)^2 X_3(\eta)$$
to derive an expansion of $\eta^{-a}+\eta^{-b}$ in $\{A_0,A_1,A_2,A_3\}$ with an error of
$\kappa_3(R)(1+\delta^{-2})|\eta-1|^8$.
We could continue to iterate this process to $k=7$ where the power of $\delta$ will grow gradually. The bound at the origin in all cases will follow from the requirement that $\Re(a+b)\le -k+1-\delta$ since
the greatest growth in the expansion will occur when $\ell=k$ and we take $\eta^{-a-b-k}$;
the bound $\Re(a+b)\ge-2k-1+\delta$ will be used later.
\end{proof}

\subsection{A regularized heat content function on the half line}\label{S3.5}
The integral of Equation~(\ref{E3.c}) converges exponentially away from the line $x=\tilde x$
but must be regularized near $x=\tilde x$ to ensure convergence.
 \begin{lemma}\label{L3.9}
Let $-1\le k\le3$, let  $(a,b)\in\mathcal{O}_k$, and $\sigma_{k,\ell}$
be as in Lemma~\ref{L3.8}. Define:\begin{eqnarray*}
&&F_{-1}(x,\tilde x;a,b):=x^{-a}\tilde x^{-b}+x^{-b}\tilde x^{-a},\nonumber\\
&&F_k(x,\tilde x;a,b):=x^{-a}\tilde x^{-b}+x^{-b}\tilde x^{-a}
-\sum_{\ell=0}^k\sigma_{k,\ell}(a,b)\{x^{-a-b-\ell}\tilde x^\ell
+x^\ell \tilde x^{-a-b-\ell}\}\text{ for }k\ge0,\\
&&h_{a,b,k}^\HL(t):=\displaystyle
\frac1{\sqrt{4\pi t}}\int_0^\infty\int_0^xF_k(x,\tilde x;a,b)e^{-(x-\tilde x)^2/(4t)}d\tilde xdx\,.
\end{eqnarray*}
\begin{enumerate}
\item The integral defining $h_{a,b,k}^\HL(t)$ converges absolutely.
\item There exists a holomorphic function $C_k(a,b)$ on $\mathcal{O}_k$
so that
$h_{a,b,k}^\HL(t)=C_k(a,b)t^{(1-a-b)/2}$.
\item $C_{-1}(a,b)=c(a,b)$ where $c(a,b)$ is as given in Theorem~\ref{T1.1}.
\end{enumerate}\end{lemma}

\begin{proof} We postpone for the moment the proof of Assertion~(1)
and proceed formally.
We replace $x$ by $\sqrt{4t}x$ and $\tilde x$ by
$\sqrt{4t}\tilde x$ and use the fact that
$F_k(\lambda x,\lambda \tilde x)=\lambda^{-a-b}F_k(x,\tilde x)$ to express
\begin{eqnarray*}
&&h_{a,b,k}^\HL(t)=\frac1{\sqrt\pi}(4t)^{(1-a-b)/2}C_k(a,b),\text{ where}\\
&&C_k(a,b):=\frac{2^{1-a-b}}{\sqrt\pi}\int_0^\infty\int_0^x
F_k(x,\tilde x;a,b)e^{-(x-\tilde x)^2}d\tilde xdx\,.
\end{eqnarray*}
Set $\tilde x=x\eta$. Let $G_k$ be as in Lemma~\ref{L3.8}. This yields:
\begin{eqnarray*}
C_k(a,b)&=&\frac{2^{1-a-b}}{\sqrt\pi}
\int_0^\infty\int_0^1x^{1-a-b}G_k(\eta;a,b)e^{-x^2(1-\eta)^2}d\eta dx
\\&=&\frac{2^{1-a-b}}{\sqrt\pi}
\int_0^1\int_0^\infty x^{1-a-b}G_k(\eta;a,b)e^{-x^2(1-\eta)^2}dxd\eta\,.
\end{eqnarray*}
Change variables $x=\xi^{1/2}(1-\eta)^{-1}$ to see
\begin{equation}\label{E3.L}
\begin{array}{l}
\displaystyle C_k(a,b)=\frac1{\sqrt\pi}2^{-a-b}
\int_0^1\int_0^\infty\xi^{(-a-b)/2}e^{-\xi}G_k(\eta;a,b)(1-\eta)^{a+b-2}d\xi d\eta\\
\qquad\hphantom{aaa}
=\displaystyle\frac1{\sqrt\pi}2^{-a-b}\Gamma\left(\frac{2-a-b}2\right)\int_0^1
G_k(\eta;a,b)(1-\eta)^{a+b-2}d\eta\,.\vphantom{\vrule heigh 13pt}
\end{array}\end{equation}
Since $-2k-1<\Re(a+b)$, Lemma~\ref{L3.8}~(1) implies
$|G_k(\eta;a,b)|$ is integrable on $[\frac12,1]$.
Lemma~\ref{L3.8}~(2) implies $|G_k(\eta;a,b)|$ is integrable
on $[0,\frac12]$.
This proves Assertion~(1) and Assertion~(2). Suppose $k=-1$ so there are
no regularizing terms. If $\Re(a+b)>1$, we use Mathematica~\cite{M9}
to complete the proof by computing:
\medbreak\hfill $C_{-1}(a,b)=\displaystyle
\frac1{\sqrt\pi}2^{-a-b}\Gamma\left(\frac{2-a-b}2\right)
\int_0^1\left\{\eta^{-a}+\eta^{-b}\right\}(1-\eta)^{a+b-2}d\eta=c(a,b)$\hfill\end{proof}

\subsection{A regularized heat content function on the interval}\label{S3.6}
Adopt the notation established previously
to define $F_k(x,y;a,b)$, $G_k(\eta;a,b)$, $c(a,b)$, and $C_k(a,b)$.
Let
$$\mathcal{E}_N(\delta,R,t)=\kappa(N,R)(1+\delta^{-4})\left\{t^{(N+1)/2}+e^{-1/( 36t)}\right\}\,.$$
\begin{lemma}\label{L3.10}
Let $-1\le k\le4$.
Let $N$ be a non-negative integer. If $(a,b)\in\mathcal{O}_k$, let
$$
h_{a,b,k}^\IN(t):=
\frac1{\sqrt{4\pi t}}\int_0^1\int_0^xF_k(x,\tilde x;a,b)e^{-(x-\tilde x)^2/(4t)}d\tilde xdx\,.
$$
Let $C_k(a,b)$ be as given in Equation~(\ref{E3.L}).
There exist functions
$C_{k,n}^1(a,b)$, $C_n^2(a,b)$, and $C_{k,n}^3(a,b)$
which are holomorphic on $\mathcal{O}$ and there exists a bound
$\kappa(N,R)$ so that if $(a,b)\in\mathcal{O}_k$, if
$|(a,b)|\le R$, and if $\operatorname{dist}((a,b),\mathcal{O}_k^c)\ge\delta$,
then
$$\big|h_{a,b,k}^\IN(t)-\big\{C_k(a,b)t^{(1-a-b)/2}
+\sum_{n=0}^NC_{k,n}^1(a,b)t^{n/2}\big\}\big|
\le\mathcal{E}_N(\delta,R,t)\,.$$
In the special case that $k=-1$, we have that $C_{-1}(a,b)=c(a,b)$ and $C_{-1,n}^1(a,b)=c_n(a,b)$.
\end{lemma}

\begin{proof}
We compare the situation on $\IN$ with the situation on $\HL$. Set
\begin{eqnarray*}
\mathcal{H}_{a,b,k}(t)&=&h_{a,b,k}^\HL(t)-h_{a,b,k}^\IN(t)\\
&=&\frac1{\sqrt{4\pi t}}\int_0^\infty\int_0^x\{1-\chi_\IN(x)\chi_\IN(\tilde x)\}
F_k(x,\tilde x;a,b)e^{-(x-\tilde x)^2/(4t)}d\tilde xdx\,.
\end{eqnarray*}
If $\tilde x\le x$, then $\chi_\IN(x)\chi_\IN(\tilde x)=\chi_\IN(x)$. Consequently
\begin{eqnarray*}
\mathcal{H}_{a,b,k}(t)&=&
\frac1{\sqrt{4\pi t}}\int_0^\infty\int_0^x\{1-\chi_\IN(x)\}
F_k(x,\tilde x;a,b)e^{-(x-\tilde x)^2/(4t)}d\tilde xdx\\
&=&\frac1{\sqrt{4\pi t}}\int_1^\infty\int_0^x
F_k(x,\tilde x;a,b)e^{-(x-\tilde x)^2/(4t)}d\tilde xdx\,.
\end{eqnarray*}
Decompose
$\mathcal{H}_{a,b,k}(t)=\mathcal{H}_{a,b,k}^{[0,\frac12]}(t)
+\mathcal{H}_{a,b,k}^{[\frac12,1]}(t)$
where
\begin{eqnarray*}
&&\mathcal{H}_{a,b,k}^{[0,\frac12]}(t):=
\frac1{\sqrt{4\pi t}}\int_1^\infty\int_{0}^{\frac12 x}
F_k(x,\tilde x;a,b)e^{-(x-\tilde x)^2/(4t)}d\tilde xdx,\\
&&\mathcal{H}_{a,b,k}^{[\frac12,1]}(t):=
\frac1{\sqrt{4\pi t}}\int_1^\infty\int_{\frac12 x}^x
F_k(x,\tilde x;a,b)e^{-(x-\tilde x)^2/(4t)}d\tilde xdx\,.
\end{eqnarray*}
We now show that the integral defining $\mathcal{H}_{a,b,k}^{[0,\frac12]}(t)$ converges
and decays exponentially as $t\downarrow0$.
Replace $x$ by $\sqrt{4t}x$ and $\tilde x$ by $\sqrt{4t}\tilde x$ and use the fact that
$F_k(cx,c\tilde x;a,b)=c^{-a-b}F(x,\tilde x;a,b)$ to see:
\begin{equation}\label{E3.m}\begin{array}{l}
\left|\mathcal{H}_{a,b,k}^{[0,\frac12]}(t)\right|
\le\displaystyle
\frac1{\sqrt{4\pi t}}\int_1^\infty\int_0^{\frac12 x}
\left|F_k(x,\tilde x;a,b)\right|e^{-(x-\tilde x)^2/(4t)}d\tilde xdx\\
\qquad\le\displaystyle\frac1{\sqrt\pi}(4t)^{\Re(1-a-b)/2}\int_{\frac1{\sqrt{4t}}}^\infty\int_0^{\frac12 x}
\left|F_k(x,\tilde x;a,b)\right|
e^{-(x-\tilde x)^2}d\tilde xdx\,.
\end{array}\end{equation}
Recall that:
\begin{eqnarray*}
&&G_k(\eta;a,b)= \eta^{-a}+\eta^{-b}-\sum_{\ell=0}^k\sigma_{k,\ell}(a,b)
\{\eta^\ell+\eta^{-a-b-\ell}\},\\
&&F_k(x,\tilde x;a,b)=x^{-a}\tilde x^{-b}+x^{-b}\tilde x^{-a}
-\sum_{\ell=0}^k\sigma_{k,\ell}(a,b)\{x^{-a-b-\ell}\tilde x^\ell
+x^\ell \tilde x^{-a-b-\ell}\}\,.
\end{eqnarray*}
We set $\tilde x=\eta x$ and change variables in Equation~(\ref{E3.m}) to see:
$$
\left|\mathcal{H}_{a,b,k}^{[0,\frac12]}(t)\right|
\le\frac1{\sqrt\pi}(4t)^{\Re(1-a-b)/2}\int_{\frac1{\sqrt{4t}}}^\infty\int_0^{\frac12}
x^{\Re(1-a-b)}\left|G_k(\eta;a,b)\right|e^{-x^2(1-\eta)^2}d\eta dx\,.
$$
By Lemma~\ref{L3.8}~(2),
$|G_k(\eta;a,b)|\le\kappa(k,R)(1+\delta^{-4})\eta^{\delta-1}$ on $[0,\frac12]$.
We may estimate
$$
e^{-x^2(1-\eta)^2}\le e^{-x^2(1-\frac12)^2}=e^{-x^2/8}e^{-x^2/8}
\le e^{-1/(32t)}e^{-x^2/8}
$$
for $0\le\eta\le\textstyle\frac12$ and $\frac1{\sqrt{4t}}\le x<\infty$.
We show that $\mathcal{H}_{a,b,k}^{[0,\frac12]}(t)$ is exponentially damped as $t\downarrow0$
by estimating:
\begin{eqnarray*}
\left|\mathcal{H}_{a,b,k}^{[0,\frac12]}(t)\right|&\le&\kappa(k,R)(1+\delta^{-4})(4t)^{\Re(1-a-b)/2}e^{-1/(32t)}
\int_{1/\sqrt{4t}}^\infty x^{\Re(1-a-b)}
e^{-x^2/8}dx\cdot\int_0^{\frac12}\eta^{\delta-1}d\eta\\
&\le&\kappa(k,R)\mathcal{E}_N(\delta,R,t)\,.
\end{eqnarray*}

To examine $\mathcal{H}_{a,b,k}^{[\frac12,1]}$, we again set $\tilde x=x\eta$ to express
\begin{eqnarray*}
\mathcal{H}_{a,b,k}^{[\frac12,1]}(t)=
\frac1{\sqrt{4\pi t}}\int_1^\infty\int_{\frac12}^1 x^{1-a-b}
G_k(\eta;a,b)e^{-x^2(1-\eta)^2/(4t)}d\eta dx\,.
\end{eqnarray*}
We expand $G_k(\eta;a,b)$ in a Taylor series about $\eta=1$ on $[\frac12,1]$.
Since $G_k(\eta;a,b)=O((1-\eta)^{2k+2})$
the first $2k+1$ terms vanish:
$$
G_k(\eta;a,b)=\sum_{n=2k+2}^N\theta_{n,k}(a,b)
(1-\eta)^n+r_N(\eta;a,b)\,,
$$
where $N\ge2k+2$ and $|r_N(\eta;a,b)|\le\kappa(k,N,R)(1+\delta^{-4})(1-\eta)^{N+1}$. We may then expand:
\begin{eqnarray}
\mathcal{H}_{a,b,k}^{[\frac12,1]}(t)&=&
\sum_{n=2k+2}^{N}\frac{\theta_{n,k}(a,b)}
{\sqrt{4\pi t}}\int_1^\infty\int_{\frac12}^1 x^{1-a-b}(1-\eta)^{n}e^{-x^2(1-\eta)^2/(4t)}
d\eta dx\label{E3.n}\\
&+&\frac1{\sqrt{4\pi t}}\int_1^\infty\int_{\frac12}^1
x^{1-a-b}r_N(\eta;a,b)e^{-x^2(1-\eta)^2/(4t)}d\eta dx\,.\label{E3.o}
\end{eqnarray}
Since $(a,b)\in\mathcal{O}_k$, $\Re(a+b)>-2k-1$. Consequently $1-\Re(a+b)<2k+2\le N$.
This shows that $|x^{1-a-b}|\le x^N$ for $x\in\HU{1}{\infty}$. We may
therefore the remainder given in Equation~(\ref{E3.o}) by:
\begin{eqnarray*}
&&\frac1{\sqrt{4\pi t}}
\int_1^\infty\int_{\frac12}^1\left|x^{1-a-b}r_N(\eta;a,b)
e^{ -x^2(1-\eta)^2/(4t)}\right|d\eta dx\\
&&\quad\le\frac{\kappa(k,N,R)}{\sqrt{4\pi t}}(1+\delta^{-4})
\int_1^\infty\int_{\frac12}^1x^N(1-\eta)^{N+1}e^{-x^2(1-\eta)^2/(4t)}
d\eta dx\\
&&\quad=\frac{\kappa(k,N,R)}{\sqrt{4\pi t}}(1+\delta^{-4})\int_{\frac12}^1
\int_1^\infty x^N(1-\eta)^{N+1}e^{-x^2(1-\eta)^2/(4t)}
dxd\eta\\
&&\quad\le\kappa(k,N,R)(1+\delta^{-4})t^{N/2}
\int_{\frac12}^1\int_0^\infty
 \xi^{(N-1)/2}e^{-\xi}d\xi d\eta\,,
\end{eqnarray*}
where we have made the change of variables $x=\xi^{1/2}(1-\eta)^{-1}\sqrt{4t}$.
The integral converges and leads to an estimate
which is bounded by $\mathcal{E}_N(\delta,R,t)$.

Fix $n\ge 2k+2$. We examine a typical term in the summation of Equation~(\ref{E3.n}). We
set $u=x(1-\eta)/\sqrt{4t}$ to express
\begin{eqnarray*}
&&\frac{\theta_{n,k}(a,b)}
{\sqrt{4\pi t}}\int_1^\infty\int_{\frac12}^1x^{1-a-b}(1-\eta)^{n}e^{-x^2(1-\eta)^2/(4t)}
d\eta dx\\
&=&\frac{\theta_{n,k}(a,b)}{\sqrt\pi}(4t)^{n/2}
\int_1^\infty\int_0^{(1-\frac12)x/\sqrt{4t}}x^{-a-b-n}u^ne^{-u^2}dudx
\,.
\end{eqnarray*}
The $dx$ integral is convergent at $\infty$ since
$$
\Re(a+b)\ge-2k-1+\delta\text{ implies }\Re(-a-b-n)\le 2k+1-(2k+2)-\delta=-1-\delta\,.
$$
Replacing the integral $0\le u\le(1-\frac12)x/\sqrt{4t}$ by $0\le u\le\infty$ introduces an exponentially small
error as $t\downarrow0$ since $x\ge1$. The existence of a series in the desired form now
follows.

If $k=-1$, there are no regularizing
terms and $G_{-1}(\eta;a,b)=\eta^{-a}+\eta^{-a}$. By Lemma~\ref{L3.9}, $C_{-1}(a,b)=c(a,b)$. Furthermore,
expanding $G_{-1}$ in a Taylor series about $\eta=1$ shows that the coefficient of $t^{n/2}$ has
the desired form since
\medbreak\hfill$\theta_{n,k}(a,b)=\frac{\theta_n(a)+\theta_n(b)}{n!}$\hfill\vphantom{.}
 \end{proof}

\subsection{The proof of Theorem~\ref{T3.1}}\label{S3.7}
Let $(a,b)\in\mathcal{O}_k$ for $k=-1,0,1,2,3$. We may expand
$$h_{a,b}^\IN(t)=h_{a,b,k}^\IN(t)+\sum_{\ell=0}^k\sigma_{k,\ell}(a,b)h_{a+b+\ell,-\ell}^\IN(t)\,.$$
By Lemma~\ref{L3.10},t $h_{a,b,k}^\IN(t)$ has an appropriate
asymptotic expansion. By Lemma~\ref{L3.7}, $h_{a+b+\ell,-\ell}^\IN(t)$ also has an
appropriate asymptotic expansion for $\ell=0,\dots,k$. Thus
$$
\left|h_{a,b}^\IN(t)-\left\{
c_k(a,b)t^{(1-a-b)/2}+\sum_{n=0}^Nc_{k,n}(a,b)t^{n/2}\right\}\right|
\le\mathcal{E}_N(\delta,R,t)\,,$$
where the functions $c_k(a,b)$ and $c_{k,n}(a,b)$ are holomorphic on $\mathcal{O}_k$.
Since $\mathcal{O}_i\cap\mathcal{O}_{i+1}$ is non-empty for $i=0,1,2,3$, we may use the identity
theorem to conclude that $c_i(a,b)=c_{i+1}(a,b)$ and $c_{i,n}=c_{i+1,n}$ for $i=0,1,2,3$ and thus
the holomorphic functions patch together properly on
$\mathcal{O}_0\cup\mathcal{O}_1\cup\mathcal{O}_2\cup\mathcal{O}_3$.
This yields expansions:
\begin{eqnarray*}
&&\left|h_{a,b}^\IN(t)-\left\{
c_{-1}(a,b)t^{(1-a-b)/2}+\sum_{n=0}^Nc_{-1,n}(a,b)t^{n/2}\right\}\right|\le\mathcal{E}_N(\delta,R,t)
\text{ on }\mathcal{O}_{-1},\\
&&\left|h_{a,b}^\IN(t)-\left\{
c_0(a,b)t^{(1-a-b)/2}+\sum_{n=0}^Nc_{0,n}(a,b)t^{n/2}\right\}\right|\le
\mathcal{E}_N(\delta,R,t)
\text{ on }\mathcal{U},\text{ where}\\
&&\mathcal{U}:=\mathcal{O}_{0}\cup\mathcal{O}_1\cup\mathcal{O}_2=\{(a,b)\}:\Re(a)<1,\ \Re(b)<1,\
-7<\Re(a+b)<1,\quad a+b\ne-1\}\,.
\end{eqnarray*}
We may now use Lemma~\ref{L3.5}
to relate $h_{a,b}^\IN$, $h_{a-2,b}^\IN$, and $h_{a-1,b-1}^\IN$. This shifts
$\mathcal{O}_{-1}\cup\mathcal{O}_0$ (which is a disconnected open set) down into
$\mathcal{O}_0\cup\mathcal{O}_1$ (which is a connected open set). This permits us to extend
$\{c_{-1},c_{-1,n}\}$ and $\{c_0(a,b),c_{0,n}\}$ across dividing hyperplane $\Re(a+b)=1$ and to remove the singularity
except at $a+b=1$. We denote the resulting extensions by $\{c_{-1},c_{-1,n}\}$ and have an asymptotic
series:
\begin{eqnarray*}
&&\left|h_{a,b}^\IN(t)-\left\{
c_{-1}(a,b)t^{(1-a-b)/2}+\sum_{n=0}^Nc_{-1,n}(a,b)t^{n/2}\right\}\right|\le
\mathcal{E}_N(\delta,R,t)\text{ for}\\
&&\Re(a)<1,\ \Re(b)<1,\ a+b\ne1,\ a+b\ne-1,\ -3<\Re(a+b)<2\,.
\end{eqnarray*}

If $\Re(a)<1$, if $\Re(b)<1$, and if $\Re(a+b)>1$, then the desired formulas
$c_{-1}(a,b)=c(a,b)$ and $c_{-1,n}(a,b)=c_n(a,b)$ follow from Lemma~\ref{L3.10}. The
formulas in general now follow by analytic continuation since we have extended $c_{-1}$ and
$c_{-1,n}$ holomorphically across the barrier $\Re(a+b)\ne-1$ omitting only the values $a+b=-1$.
This proves Assertion~(1) of Theorem~\ref{T3.1}
by giving a uniform estimate for the remainder if $(a,b)\in\mathcal{O}$ and $\Re(a+b)>-7$.

We  use Lemma~\ref{L3.4} to establish Assertion~(2) of Theorem~\ref{T3.1}. We
need to introduce cut-off functions.
Let $\Xi$ satisfy Equation~(\ref{E3.e}). Let $\Re(a)<1$, $\Re(b)<1$, and $a+b\ne1,-1,-3,\dots$.
We may assume $\Re(a+b)\le-7$.
Let $\phi(x):=x^{-a}\Xi(x)$ and $\rho(x):=x^{-b}\Xi(x)$.
We define heat content functions for $[0,1]$ in $M=\mathbb{S}^1$ or in $M=\mathbb{R}$:
$$
\begin{array}{l}
\beta_{a,b,\Xi}(M,t):=\beta_\IN(x^{-a},x^{-b},D_M),\\
\beta_{a,b,\Xi;1}(M,t):=\beta_\IN(\Xi x^{-a},\Xi x^{-b},D_M)(t),\\
\beta_{a,b,\Xi;2}(M,t):=\beta_\IN((1-\Xi)x^{-a},\Xi x^{-b},D_M)(t),\\
\beta_{a,b,\Xi;3}(M,t):=\beta_\IN(\Xi x^{-a},(1-\Xi)x^{-b},D_M)(t),\\
\beta_{a,b,\Xi;4}(M,t):=\beta_\IN((1-\Xi)x^{-a},(1-\Xi)x^{-b},D_M)(t)\,.
\end{array}$$

\medbreak\noindent
We may choose the notation so that $\Re(a)\le\Re(b)$ and hence $\Re(a)\le-\frac72$. This
implies $\Re(a)<0$, $\Re(a+1)<0$, and $\Re(a+2)<0$. We now use
Lemma~\ref{L3.3} to obtain the following
inequalities for $1\le i\le 4$:
$$
\left|\beta_{\tilde a,b,\Xi;i}(S^1,t)-\beta_{a,b,\Xi;i}(\mathbb{R},t)\right|\le\kappa(R)\frac1{1-\Re(b)}e^{-1/(8t)}
\text{ for }\tilde a=a,a+1,a+2\,.
$$
Since $(1-\Xi(x))x^{-a}$ is smooth on $\IN$,
Theorem ~\ref{T1.4} provides the existence of
a suitable asymptotic series in $t^{n/2}$ for $\beta_{a,b,\Xi;2}(S^1,t)$ and $\beta_{a,b,\Xi;4}(S^1,t)$. Similarly, since $(1-\Xi(x))x^{-b}$ is smooth on $\IN$, we may conclude that
$\beta_{a,b,\Xi;3}(S^1,t)$
has a suitable asymptotic series as well. Consequently $\beta_{a,b}(S^1,t)$ has a suitable
asymptotic series if and only if $\beta_{a,b,\Xi;1}(S^1,t)$ has such a series.
Since $\Xi(x) x^{-a}\in C_0^1[0,1]$, we may apply Lemma~\ref{L3.4}~(2),
$$\beta_{a,b,\Xi;1}(t)=\beta_{a,b,\Xi;1}(0)+
\int_0^t\beta_\IN((\Xi x^{-a})^{\prime\prime},\Xi x^{-b},D_{S^1})(s)ds\,.$$
We expand
$$(\Xi(x)x^{-a})^{\prime\prime}=a(a+1)\Xi(x) x^{-a-2}-2a\Xi^\prime(x) x^{-a-1}
+\Xi^{\prime\prime}(x)x^{-a}\,.$$
By induction, we will have an asymptotic series for $\phi(x)=\Xi(x) x^{-a-2}$;
 the existence of a suitable
series for $\phi(x)=\Xi^\prime(x) x^{-a-1}$ or
$\phi(x)=\Xi^{\prime\prime}(x)x^{-a}$ follows from
Theorem~\ref{T1.2}
since these functions are smooth and are compactly supported in $(0,1)$.
(It is that stage that our estimates become
non-uniform in $(a,b)$ since we did not take care in this regard in the proof of
Theorem~\ref{T1.2}. This defect could be easily remedied but it does not seem worth
the additional technical fuss involved). We integrate the asymptotic series to establish
Theorem~\ref{T3.1} in complete generality; again, analytic continuation
yields the values of the coefficients. This completes the proof of Theorem~\ref{T3.1}.

The following is a corollary to the arguments given above:

\begin{corollary}\label{L3.11}
Let $(a,b)\in\mathcal{O}$ and let $\Xi$ be a cut-off function satisfying Equation~(\ref{E3.e}).
There exist functions $c_{a,b,\Xi}$ and $\tilde c_{a,b,j,\Xi}$ which are holomorphic on $\mathcal{O}$
so that if $N>\Re(-a-b)+3$, then:
$$\beta_\IN(\Xi x^{-a},\Xi x^{-b},D_{\mathbb{R}})(t)
=\sum_{n=0}^Nc_{a,b,n,\Xi}t^{n/2}
+c(a,b)t^{(1-a-b)/2}+O(t^{(N+1)/2})\text{ as }t\downarrow0\,.
$$
\end{corollary}

\subsection{The log terms when $a+b=1$}\label{S3.8}
We now use the uniform estimates of $\ref{T3.1}$ near the hyperplane $\Re(a+b)=1$.
we must restrict to $(a,b)\in\mathbb{R}^2$ since we will use certain monotonicity results
that fail if $a$ and $b$ can be complex.

Equation~(\ref{E3.b}) shows that $c_0$ and $c(a,b)$ have simple poles when $a+b=1$. This
gives rise to a logarithmic term in the asymptotic
expansion at these values when the associated terms in the asymptotic expansion collide.
We begin with a technical result:
\begin{lemma}\label{L3.12}
Let $\mathcal{X}(a,b):=(1-a-b)c(a,b)$.
\begin{enumerate}
\item $\mathcal{X}(a,b)$ extends holomorphically to the plane $a+b=1$ by setting
$\mathcal{X}(a,1-a)=-1$.
\item Let $\mathcal{Y}(a):=
\partial_\delta\{\mathcal{X}(a+\delta,1-a)\}|_{\delta=0}$. Let $k\ge3$.
If $a\in[\epsilon,1-\epsilon]$
for some $\epsilon>0$, if $|\delta|<\min\{\epsilon/2,t\}$, and if $t\le t_0(\epsilon)$, then:
$$\left|\left\{c(a+\delta,1-a)t^{-\delta/2}-\frac1\delta\right\}-
\left\{-\frac12\ln t+\mathcal{Y}(a)\right\}\right|\le C(\epsilon)\delta(1+\ln(t)^2).
$$
\end{enumerate}\end{lemma}

\begin{proof}
We use the recursion relation $(1-a-b)\Gamma(a+b-1)=-\Gamma(a+b)$ to express
$$\mathcal{X}(a,b)=
-2^{-a-b}\frac1{\sqrt\pi}\Gamma\left(\frac{2-a-b}2\right)
 \Gamma(a+b) \cdot\left(\frac{\Gamma(1-a)}{
 \Gamma(b)}+\frac{\Gamma(1-b)}
{ \Gamma(a)}\right)\,.
$$
This is no longer singular when $a+b=1$ and we have $\mathcal{X}(a,1-a)=-1$. This
verifies Assertion~(1). We have by definition that
\begin{eqnarray*}
&&c(a+\delta,1-a)t^{-\delta/2}-\frac1\delta=
-\left\{\frac{\mathcal{X}(a+\delta,1-a)t^{-\delta/2}-\mathcal{X}(a,1-a)t^0}\delta\right\}\\
&=&-\mathcal{X}(a+\delta,1-a)\left\{\frac{t^{-\delta/2}-1}\delta\right\}
-\left\{\frac{\mathcal{X}(a+\delta,1-a)-\mathcal{X}(a,1-a)}\delta\right\}\,.
\end{eqnarray*}
We have by definition that
\begin{eqnarray*}
-\left\{\frac{\mathcal{X}(a+\delta,1-a)-\mathcal{X}(a,1-a)}\delta\right\}
=-\mathcal{Y}(a)+O(\delta)
\end{eqnarray*}
where the remainder may be estimated uniformly if $\delta$ is small and if $a$ has a compact
parameter range. The desired estimate for this error follows. Since
$-\mathcal{X}(a+\delta,1-a)=1+O(\delta)$, it suffices to examine
$$\frac{e^{-\frac{\delta}{2}\ln (t)}-1}
{\delta}=\sum_{n=1}^\infty\frac1{n!}\left(-\frac12\ln (t)\right)^n\delta^{n-1}
=-\frac12\ln(t)-\left(\frac12\ln(t)\right)^2\delta\sum_{n=2}^\infty\frac{(-\frac12\delta\ln(t))^{n-2}}{n!}.
$$
Since $|\delta|<t$, we can choose $t$ sufficiently small so that $|\delta\ln(t)|<|t\ln t|<1$
so the infinite series can be uniformly bounded by $e$. This proves Assertion~(2).
\end{proof}

\begin{lemma}\label{L3.13} If $a+b=1$, $a\in(0,1)$, if $b\in(0,1)$, and if $k\ge2$, then
$$\left
|h_{a,1-a}^\IN(t)-\left\{-\frac12\ln t+\mathcal{Y}(a)
+\sum_{n=1}^kc_n(a,1-a)t^{n/2}\right\}\right|
\le O(t^{(k+1)/2})\,.$$
\end{lemma}

\begin{proof} Let $k\ge 2$ be given. Choose $N$ so that $\frac N{10}\ge
2+\frac{k+1}2$.
We first establish the estimate:
\begin{equation}\label{E3.p}
\left
|h_{a,1-a}^\IN(t)-\left\{-\frac12\ln t+\mathcal{Y}(a)
+\sum_{n=1}^Nc_n(a,1-a)t^{n/2}\right\}\right|
= O(t^{N/10}\ln(t)^2)\,.
\end{equation}
We use Theorem~\ref{T3.1} to see
$$
h_{a+\delta,1-a}^\IN(t)=c(a+\delta,1-a)t^{-\delta/2}+\frac1{1-a-b}+\dots
$$
Let $a\in[\epsilon,1-\epsilon]$ and let $0<\delta<\epsilon/2$. We estimate from above:
\begin{eqnarray}
&&h_{a,1-a}^\IN(t)-\sum_{n=1}^Nc_n(a,1-a)t^{n/2}+\frac12\ln t-\mathcal{Y}(a)\nonumber\\
&\le&h_{a+\delta,1-a}^\IN(t)-\sum_{n=1}^Nc_n(a,1-a)t^{n/2}+\frac12\ln t-\mathcal{Y}(a)\nonumber\\
&=&h_{a+\delta,1-a}^\IN(t)-\sum_{n=0}^Nc_n(a+\delta,1-a)t^{n/2}+c(a+\delta,1-a)t^{-\delta/2}
\label{E3.q}\\
&+&\sum_{n=1}^N\{c_n(a+\delta,1-a)-c_n(a,1-a)\}t^{n/2}\label{E3.r}\\
&+&\left\{\frac12\ln t-\mathcal{Y}(a)+c(a+\delta,1-a)t^{-\delta/2}+\frac1\delta\right\}
\label{E3.s}\end{eqnarray}
We take $\delta=t^{\gamma}$. We use Theorem~\ref{T3.1} to estimate the terms
in Equation~(\ref{E3.q}) by
$$C(N,\epsilon)(1+t^{-4\gamma})(t^{N/2}+e^{-1/(36t)})\le
C(N,\epsilon)t^{-4\gamma + N/2}\text{ for }
t\le t_0(N,\epsilon)\,.$$
The terms $c_n$ for $1\le n\le N$ are regular at $a+b=1$. Thus we may estimate
$$|c_n(a+\delta,1-a)-c_n(a,1-a)|\le C(N,\epsilon)\delta=C(N,\epsilon)t^{\gamma}\text{ for }1\le n\le N,$$
and thereby control the terms in Equation~(\ref{E3.r}). In order to control the terms in
Equation~(\ref{E3.s}) we use Lemma~\ref{L3.12}, and henceforth require $\gamma\ge 1$. Combining all three remainders provides an upper
estimate $$ C(\epsilon)t^{\gamma}\ln(t)^2+C(N,\epsilon)t^{-4\gamma + N/2}+C(N,\epsilon)t^{\gamma}.$$ We now choose $\gamma=N/10$ as to minimize this upper estimate. This gives
for $N\ge 10$, $O(t^{N/10}\ln(t)^2).$ Reversing the sign of $\delta$ provides the desired lower estimate and establishes Equation~(\ref{E3.p}).

Since $t\log(t)$ tends to $0$
as $t\rightarrow0$ and since $N/10\ge \frac{k+1}2+2$,
$$t^{N/10}\ln(t)^2\le t^{\frac{k+1}2}\quad\text{ as }\quad t\downarrow0$$
and consequently Equation~(\ref{E3.p})
yields:
\begin{equation}\label{E3.x}
\left
|h_{a,1-a}^\IN(t)-\left\{-\frac12\ln t+\mathcal{Y}(a)
+\sum_{n=1}^kc_n(a,1-a)t^{n/2}+\sum_{n=k+1}^Nc_n(a,1-a)t^{n/2}\right\}\right|
= O(t^{\frac{k+1}2})\,.
\end{equation}
The Lemma now follows since
\medbreak
\hfill
$\displaystyle\left|\sum_{n=k+1}^Nc_n(a,1-a)t^{n/2}\right|=O(t^{\frac{k+1}2})\,.$\hfill\vphantom{.}
\end{proof}

\subsection{The proof of Theorem~\ref{T3.2}}\label{S3.9}
Theorem~\ref{T3.2} follows from Lemma~\ref{L3.13} if $a+b=1$. Suppose first $a+b=-1$ so $k=1$.
Let $\phi(x)=x^{-a}\Xi_1(x)$ and let $\rho(x)=x^{-b}\Xi_2(x)$. We may suppose $a\le b$.
There are 3 cases:
\begin{enumerate}
\item $0<b<1$. As $a<-1$, we may apply Lemma~\ref{L3.4}~(2) to see:
\begin{eqnarray*}
&&\beta_\IN(\phi,\rho,D_{S^1})(t)=\beta_\IN(\phi,\rho,D_{S^1})(0)
+\int_0^t\beta(\phi^{\prime\prime},\rho,D_{S^1})(s)ds\\
&&\qquad=\beta_\IN(\phi,\rho,D_{S^1})(0)+a(a+1)
\int_0^t\beta_\IN(x^{-(a+2)}\Xi_1,x^{-b}\Xi_2,D_{S^1})(s)ds\\
&&\qquad\quad+\int_0^t\
\beta_\IN(-2ax^{-a-1}\Xi_1^\prime+x^{-a}\Xi^{\prime\prime},x^{-b}\Xi_2,D_{S^1})(s)ds\,.
\end{eqnarray*}
Since $\{-2ax^{-a-1}\Xi_1^\prime+x^{-a}\Xi^{\prime\prime}\}$ is smooth compactly
supported in $(0,1)$, we use Theorem~\ref{T1.3} to obtain a suitable asymptotic series
for $\beta(-2ax^{-a-1}\Xi_1^\prime+x^{-a}\Xi^{\prime\prime},x^{-b}\Xi_2,D_{S^1})(s)$
which can be then integrated term by term to get an appropriate full asymptotic series in $t$. We integrate
the asymptotic series of Theorem~\ref{T1.2} for $(a+2)+b=1$ for
$\beta(x^{-(a+2)}\Xi_1,x^{-b}\Xi_2,D_{S^1})(s)$ to establish the desired
series in the setting:
$$
\beta_\IN(\phi,\rho,D_{S^1})(t)=\beta_{0,1}+\beta_{1,1}t-\frac{a(a+1)}2t\log(t)+\dots
$$
\item $b<0$ and $a<0$. We apply Lemma~\ref{L3.4}~(3) to see
$$
\beta_\IN(\phi,\rho,D_{S^1})(t)=\beta_{0,1}+\beta_{1,1}t+\frac{ab}2t\log(t)+\dots
$$
Note the sign change. We now set $b=-1-a$ to see the coefficient of $t\log t$ is
$-\frac{a(a+1)}2$ as desired.
\item $b=0$ and $a=-1$. In this case the argument given above fails since Lemma~\ref{L3.4}
is not applicable. Instead
we use Theorem~\ref{T1.4} to conclude
 there are no log terms and we have a full asymptotic series; the coefficient $a(a+1)=0$
 in this setting.
\end{enumerate}

We now proceed by induction and suppose that an appropriate asymptotic series
has been established if $\tilde a+\tilde b=1-2k\le -1$. Suppose that $a+b=1-2k-2\le-3$.
We suppose that $a\le b$ and thus $2a\le-3$ so $a\le-3/2$. We apply the
argument above and use the existence of the asymptotic series for $(a+2,b)$.
The coefficient of $t^k\log(t)$ in
$\beta_\IN(\phi^{\prime\prime},\rho,D_{S^1})(t)$ is
$$-\frac{(a+2)(a+3)...(a+2+2k-1)}{2\cdot k!}$$
and thus the coefficient of $t^{k+1}\log(t)$ in $\beta_\IN(\phi,\rho,D_{S^1})(t)$ is
$$-\frac{(a+2)((a+2)+1)...((a+2)+2k-1)}{2\cdot k!}\cdot\frac{a(a+1)}{k+1}=
-\frac{a(a+1)...(a+2(k+1)-1)}{2\cdot(k+1)!}\,.$$
The desired result now follows.\hfill\qed

Again, the following is an immediate consequence of the arguments we have given above.

\begin{lemma}\label{L3.14}
Let $(a,b)\in\mathbb{R}^2$ satisfy $\Re(a)<1$, $\Re(b)<1$, and $a+b=1-2k$ where $k$ is
a non-negative integer. Let $\Xi$ be a cut-off function satisfying Equation~(\ref{E3.e}).
There exist functions $c_{a,b,n,\Xi}$
so that if $N>-\Re(a+b)+3$, then
$$\beta_\IN(\Xi x^{-a},\Xi x^{-b},D_{\mathbb{R}})(t)
=\sum_{n=0}^Nc_{a,b,n,\Xi}t^{n/2}
-\frac{a(a+1)...(a+2(k+1)-1)}{2\cdot(k+1)!}\log(t)t^k+O(t^{(N+1)/2})
$$
as $t\downarrow0$.
\end{lemma}

\section{The pseudo-differential calculus}\label{S4}

In Section~\ref{S4}, we will use the calculus of pseudo-differential
to complete the proof of Theorem~\ref{T1.1}; the special case computation
of Theorem~\ref{T3.1} is an essential ingredient. Throughout this section, we assume
the ambient Riemannian manifold $(M,g)$ is compact and without boundary,
and that $\Omega$ is a compact smooth subdomain of
$M$ with smooth
boundary $\partial\Omega$. We assume that $(a,b)\in\mathcal{O}$
where $\mathcal{O}$ is as defined in Equation~(\ref{E1.b}). We assume that
$\phi$ and $\rho$ are smooth on
the interior of $\Omega$ and that $r^a\phi$ and $r^b\phi$ are smooth near
the boundary of $\Omega$.

By using a partition of unity, we may assume that
$\rho$ and $\phi$ are supported within coordinate systems. Since
the kernel of the heat equation decays exponentially in
$t$ for $\operatorname{dist}_g(x,\tilde
x)\ge\epsilon>0$, we may assume that $\rho$ and $\phi$ have
support within the same coordinate system. There will, of course,
be three different types of coordinate systems to be considered -
those which touch the boundary of $\Omega$, those which are
contained entirely within the interior of $\Omega$, and those
which are contained in the exterior of $\Omega$; those contained
in the exterior of $\Omega$ play no role as they contribute an
exponentially small error as $t\downarrow0$. In
Section~\ref{S4.1} we establish notational conventions and prove a
technical result.
In Section~\ref{S4.2} we use the pseudo-differential calculus
to construct an approximation to the resolvent to the kernel
of the heat equation that we
use to begin our study of the heat content.
In Section~\ref{S4.3}, we shall examine
 coordinate systems contained in the interior of $\Omega$.
We complete
 the proof of Theorem~\ref{T1.1} in Section~\ref{S4.4}.

\subsection{Notational conventions}\label{S4.1} Let $x=(x^1,\dots,x^m)\in\mathbb{R}^m$
be coordinates on an
open set $\mathcal{U}\subset M$.
Let $(x,\xi)$ be the induced coordinate system on the cotangent space
 $T^*(\mathcal{U})$ where
we expand a $1$-form $\omega\in T^*(\mathcal{U})$ in the form
$\omega=\xi_idx^i$ to define the dual coordinates $\xi=(\xi_1,\dots,\xi_m)$.
We let $x\cdot\xi$ be the natural Euclidean pairing
$x\cdot\xi:=x^i\xi_i$.
If $\alpha=(a_1,\dots,a_m)$ is a multi-index, set
$$\begin{array}{llll}
|\alpha|=a_1+\dots+a_m,&\alpha!=a_1!\dots a_m!,&
\partial_x^\alpha:=\partial_{x_1}^{a_1}\dots\partial_{x_m}^{a_m},\\
\phi^{(\alpha)}:=\partial_x^\alpha\phi,&
\rho^{(\alpha)}:=\partial_x^\alpha\rho,&D_x^\alpha:=\sqrt{-1}^{|\alpha|}d_x^\alpha,\vphantom{\vrule height 14pt}\\
d_\xi^\alpha=\partial_{\xi_1}^{a_1}\dots\partial_{\xi_m}^{a_m},&x^\alpha=x_1^{a_1}\dots x_m^{a_m},&
\xi^\alpha:=\xi_1^{a_1}\dots\xi_m^{a_m}.\vphantom{\vrule height 14pt}
\end{array}$$
For example, with these notational conventions, Taylor's theorem becomes
\begin{equation}\label{E4.a}
f(x)=\sum_{|\alpha|\le n}\textstyle\frac1{\alpha!}f^{(\alpha)}(x_0)(x-x_0)^\alpha+O(|x-x_0|^{n+1})\,.
\end{equation}
Let $d\nu_x$, $d\nu_{\tilde x}$, and $d\nu_\xi$ be the usual
Euclidean Lebesgue measure on $\mathbb{R}^m$.
Let $L^2_e$ denote $L^2(\mathbb{R}^m)$
with respect to Lebesgue measure. Let $g(x)=g_{ij}(x)dx^i\circ dx^j$ be
a Riemannian metric on $\mathcal{U}$ and define:
$$||\xi||_{g^*(x)}^2:=g^{ij}(x)\xi_i\xi_j\text{ and }||x-\tilde x||_{g(x)}^2:=g_{ij}(x)(x^i-\tilde x^i)(x^j-\tilde x^j)\,.$$
We shall always be restricting to compact $x$ and $\tilde x$ subsets. Let
$(\cdot,\cdot)_e$ and $||\cdot||_e$ denote the usual Euclidean inner product and norm, respectively:
$$(x,y)_e:=x_1y_1+\dots+x_my_m\text{ and }
||(x_1,\dots,x_m)||_e^2:=(x,x)_e=x_1^2+\dots+x_m^2\,.$$
We have estimates:
$$C_1||\xi||_e^2\le ||\xi||_{g^*(x)}^2\le C_2||\xi||_e^2
\text{ and }C_1||x-\tilde x||_e^2\le ||x-\tilde x||_{g(x)}^2\le C_2||x-\tilde x||_e^2
$$
for positive constants $C_i$. Let $\theta=\sqrt{g}$; $\theta$
is a symmetric metric so that $\theta_{ij}\theta_{jk}=g_{ik}$. We then have
$$ ||\xi||_{g^*(x)}^2=||\theta^{-1}(x)\xi||_e^2\text{ and }
||(x-\tilde x)||_{g(x)}^2=||\theta(x)(x-\tilde x)||_e^2\,.$$
Use $\theta^2=g$ to lower indices and regard $\theta^2(x)(x-\tilde x)$ as a vector
$(X-\tilde X)_i:=g_{ij}(x-\tilde x)^j$.
\begin{lemma}\label{L4.1}
$||\xi||_{g^*(x)}^2+\sqrt{-1}(x-\tilde x)\cdot\xi/\sqrt t
=||\xi+\textstyle\frac12\sqrt{-1}(X-\tilde X)/\sqrt
t||_{g^*(x)}^2+||(x-\tilde x)||_{g(x)}^2/(4t)$.
\end{lemma}

\begin{proof} We expand:
\medbreak\quad $||\xi||_{g^*(x)}^2+\sqrt{-1}(x-\tilde
x)\cdot\xi/\sqrt{t}$ \medbreak\qquad $=
(\theta^{-1}(x)\xi,\theta^{-1}(x)\xi)_e+\sqrt{-1}(\theta^{-1}(x)\theta^2(x)(\tilde
x-x)/\sqrt t,\theta^{-1}(x)\xi)_e$ \medbreak\qquad
$=(\theta^{-1}(x)\{\xi+\textstyle\sqrt{-1}\theta^{2}(x)(\tilde
x-x)/\sqrt t\},\theta^{-1}(x)\xi)_e$ \medbreak\quad
$=||\theta^{-1}(x)\{\xi+\frac12\sqrt{-1}\theta^{2}(x)(\tilde
x-x)/\sqrt t\}||_e^2 +||\theta(x)(\tilde x-x)||_e^2/(4t)$
\medbreak\qquad $=||\xi+\frac12\sqrt{-1}(X-\tilde X)/\sqrt
t||_{g^*(x)}^2+||(x-\tilde x)||_{g(x)}^2/(4t)$.
\end{proof}

\subsection{The resolvent and the heat content}\label{S4.2}
 If $P$ is a pseudo-differential operator with symbol $p(x,\xi)$, then
$P$ is characterized by following identity for all $\phi\in C_0^\infty(V)$ and $\rho\in C_0^\infty(V^*)$:
\begin{equation}\label{E4.b}
\langle P\phi,\rho\rangle_{L_e^2}=(2\pi)^{-m}
\iiint e^{-\sqrt{-1}(x-\tilde x)\cdot\xi}\langle p(x,\xi)\phi(x),\rho(\tilde x)\rangle d\nu_xd\nu_\xi d\nu_{\tilde x}\,.
\end{equation}
The integrals in question here are iterated integrals - the convergence is not absolute
and the $d\nu_x$ integral has to be performed
before the $d\nu_\xi$ integral. However, if $p(x,\xi)$ decays rapidly enough in $\xi$,
then the integrals are in fact absolutely convergent
and we can interchange the order of integration to see following  \cite[Lemma 1.2.5]{G94}
that $P$ is given by a kernel:
\begin{equation}\label{E4.c}
\begin{array}{l}
\displaystyle\langle P\phi,\rho\rangle_{{L_e^2}}
=\int\langle K_P(x,\tilde x)\phi(x),\rho(\tilde x)\rangle d\nu_xd\nu_{\tilde x},
\text{ where}\\
\displaystyle K_P(x,\tilde x):=(2\pi)^{-m}\int e^{-\sqrt{-1}(x-\tilde x)\cdot\xi}p(x,\xi)d\nu_\xi\,.
\end{array}\end{equation}

Let $D_M$ be an operator of Laplace type on $C^\infty(V)$ over
$M$. This means that in a system of local coordinates $(x^1,\dots,x^m)$ on an open
subset $\mathcal{U}$ of $M$ and relative to a local frame for $V$ that $D$
has the form given in Equation~(\ref{E1.a}), i.e. after changing notation slightly that we have:
$$
D_M=a_2^{ij}(x)D_{x_i}D_{x_j}+a_1^i(x)D_{x_i}+a_0(x)\,.
$$
We ensure that Equation~(\ref{E4.b}) defines the operator $D_M$
by defining:
\begin{equation}\label{E4.d}
\begin{array}{l}
p(x,\xi)=p_2(x,\xi)+p_1(x,\xi)+p_0(x,\xi),\text{ where }\\
p_2(x,\xi):=|\xi|_{g^*(x)}^2,\quad p_1(x,\xi):=a_1^i(x)\xi_i,\quad p_0(x,\xi):=a_0(x)\,.
\vphantom{\vrule height 12pt}\end{array}\end{equation}

Let $\mathcal{R}\subset\mathbb{C}$ be the complement of a cone of
some small angle about the positive real axis and a
ball of large radius about the origin where $\mathcal{R}$ is chosen
so that $D_M$ has no eigenvalues in $\mathcal{R}$. We
let $\lambda\in\mathcal{R}$ henceforth. Following the discussion
of \cite[Lemma~1.7.2]{G94}, we define $r_n(x,\xi;\lambda)$ for
$(x,\xi)\in T^*(\mathcal{U})$ and $\lambda\in\mathcal{R}$
inductively by setting:
\begin{equation}\label{E4.e}
\begin{array}{l}
r_0(x,\xi;\lambda):=(|\xi|^2_{g^*(x)}-\lambda)^{-1},\vphantom{A_{A_{A_{A_A}}}}\\
r_n:=-r_0\displaystyle\sum_{|\alpha|+2+j-k=n,j<n}
 {\textstyle\frac1{\alpha!}}\partial_\xi^\alpha p_k\cdot D_x^\alpha r_j\text{ for }n>0\,.
\end{array}\end{equation}
Define
$$\begin{array}{lll}
\operatorname{ord}(\partial_x^\alpha p_2)
=|\alpha|,&\operatorname{ord}(\partial_x^\alpha p_1)=|\alpha|+1,&
\operatorname{ord}(\partial_x^\alpha p_0|)=|\alpha|+2,\\
\operatorname{weight}(\lambda)=2,&\operatorname{weight}(\xi)
=1.\vphantom{\vrule height 12pt}
\end{array}$$
The following lemma follows immediately by induction from
the recursive definition in Equation~(\ref{E4.e}):
\begin{lemma}\label{L4.2}
 $r_n$ is homogeneous of order $n$ in the derivatives of the symbol of $D_M$
with weight $-n-2$ in $(\xi,\lambda)$.
There exist polynomials $r_{n,j,\alpha}(x,D_M)$ for $n\le j\le 3n$ which are
homogeneous of order $n$ in the derivatives of the symbol of $D_M$
so:
$$
r_n(x,\xi;\lambda)=\sum_{2j-|\alpha|=n}r_{n,j,\alpha}(x,D_M)
(|\xi|_{g^*(x)}^2-\lambda)^{-j-1}\xi^\alpha\,.
$$
\end{lemma}

We use Equation~(\ref{E4.b}) to define the pseudo-differential operator
$R_n(\lambda)$ with symbol $r_n$
so that
$$\langle R_n(\lambda)\phi,\rho\rangle_{L^2_e}
=(2\pi)^{-m}\iiint e^{-\sqrt{-1}(x-\tilde x)\cdot\xi}\langle r_n(x,\xi;\lambda)\phi(x),
\rho(\tilde x)\rangle d\nu_xd\nu_\xi d\nu_{\tilde x}\,.$$
 Let $||_{-k,k}$ be the norm of a map from the Sobolev space
 $H_{-k}$ to the Sobolev space $H_k$.
By \cite[Lemma~1.7.3]{G94} we have that if $\lambda\ge\lambda(k)$
and if $n\ge n(k)$, then:
$$
 ||(D_M-\lambda)^{-1}-R_0(\lambda)-\dots-R_n(\lambda)||_{-k,k}\le C_k(1+|\lambda|)^{-k}\,.
$$

Orient the boundary $\gamma$ of $\mathcal{R}$ suitably.
We use the operator valued Riemann integral to define
$$e^{-tD_M}:=\frac1{2\pi\sqrt{-1}}\int_\gamma e^{-t\lambda}(D_M-\lambda)^{-1}d\lambda\,.$$
We use \cite[Lemma~1.7.5]{G94} to see that this is the fundamental solution of the heat equation
and belongs to $\operatorname{Hom}(H_{-k},H_k)$ for any $k$. We now let
\begin{equation}\label{E4.f}
e_n(x,\xi;t)=\frac1{2\pi\sqrt{-1}}\int_\gamma e^{-t\lambda}r_n(x,\xi;\lambda)d\lambda
\end{equation}
define the pseudo-differential operator
\begin{equation}\label{E4.g}
E_n(t,D_M):=\frac1{2\pi\sqrt{-1}}\int_\gamma
e^{-t\lambda}R_n(\lambda)d\lambda\,.
\end{equation}
We use Lemma~\ref{L4.2}~(3) and Cauchy's integral formula to rewrite Equation~(\ref{E4.f}) as:
\begin{equation}\label{E4.h}
e_n(x,\xi;t)=\sum_{2j-|\alpha|=n}\frac{t^j}{j!}\xi^\alpha
e^{-t|\xi|_{g^*(x)}^2}r_{n,j,\alpha}(x,D_M)\,.
\end{equation}
By Equation~(\ref{E4.c}) and Equation~(\ref{E4.h}),
the operator $E_n$ of Equation~(\ref{E4.g}) is given by the smooth kernel
\begin{equation}\label{E4.i}
K_n(x,\tilde x;t)
:=\sum_{2j-|\alpha|=n}(2\pi)^{-m}\frac{t^j}{j!}\int_{\mathbb{R}^m}
e^{-t|\xi|_{g^*(x)}^2-\sqrt{-1}(x-\tilde x)\cdot\xi}\xi^\alpha
r_{n,j,\alpha}(x,D_M)d\nu_\xi\,.
\end{equation}
Let $||\cdot||_{C^k}$ denote the $C^k$ norm. Given any $k\in\mathbb{N}$, there exists $n(k)$ so that
if $n\ge n(k)$ and if $0<t<1$, then \cite[Lemma 1.8.1]{G94} implies:
$$
||e^{-tD_M}-\sum_{n=0}^{n(k)}E_n(t,D_M)||_{-k,k}\le C_kt^k
$$
This gives rise to a corresponding estimate (after increasing $n(k)$ appropriately):
\begin{equation}\label{E4.j}
||K(t,x,\tilde x,D_M)-\sum_{n=0}^{n(k)}K_n(t,x,\tilde x,
D_M)||_{C^k}\le C_kt^k\,.
\end{equation}

We use Equation~(\ref{E4.h}), Equation~(\ref{E4.i}), and Equation~(\ref{E4.j}) to expand
\begin{eqnarray*}
&&\beta(\phi,\rho,D_M)(t)=\sum_{2j-|\alpha|=0}^{n(k)}(2\pi)^{-m}\frac{t^j}{j!}\iiint
e^{-t|\xi|_{g^*(x)}^2-\sqrt{-1}(x-\tilde x)\cdot\xi}\xi^\alpha\\
&&\qquad\qquad\qquad\qquad\times \langle
r_{n,j,\alpha}(x,D_M)\phi(x),\rho(\tilde x)\rangle
 d\nu_xd\nu_\xi d\nu_{\tilde x}+O(t^k)\,.
\end{eqnarray*}
We examine a typical term in the sum setting:
\begin{eqnarray*}
&&\beta_{n,j,\alpha}(\phi,\rho)(t):=(2\pi)^{-m}\frac{t^j}{j!}\iiint e^{-t|\xi|_{g^*(x)}^2-\sqrt{-1}(x-\tilde x)\cdot\xi}\xi^\alpha\\
&&\qquad\qquad\qquad\qquad\times\langle r_{n,j,\alpha}(x)\phi(x),\rho(\tilde x)\rangle d\nu_xd\nu_\xi d\nu_{\tilde x}\,.
\end{eqnarray*}
Here all integrals are over $\mathbb{R}^m$ and converge absolutely for $t>0$; $\phi$ and $\rho$ have compact support.
We change variables setting $\tilde\xi:=t^{1/2}\xi$ to
express:
\begin{eqnarray*}
&&\beta_{n,j,\alpha}(\phi,\rho)(t)=\frac{t^{j-\frac12m-\frac12|\alpha|}}{j!}(2\pi)^{-m}
\iiint e^{-|\tilde\xi|_{g^*(x)}^2-\sqrt{-1}(x-\tilde x)\cdot\tilde\xi/\sqrt{t}}\tilde\xi^\alpha\\
&&\qquad\qquad\qquad\qquad\times\displaystyle \langle
r_{n,j,\alpha}(x,D_M)\phi(x),\rho(\tilde x)\rangle
d\nu_xd\nu_{\tilde\xi}d\nu_{\tilde x}\,.\nonumber
\end{eqnarray*}
Note that $\frac12n=j-\frac12|\alpha|$. We adopt the notation of Lemma~\ref{L4.1} and
make a complex change of coordinates
setting:
$$\textstyle\eta=\tilde\xi+\frac12\sqrt{-1}(X-\tilde X)/\sqrt t\,.$$
We then apply Lemma~\ref{L4.1} and the binomial theorem to express:
\begin{eqnarray*}
&&\beta_{n,j,\alpha}(\phi,\rho,)(t)\\
&=&(2\pi)^{-m}\sum_{\alpha_1+\alpha_2=\alpha}
{\textstyle\frac{\alpha!}{j!\alpha_1!\alpha_2!}}(-\sqrt{-1})^{|\alpha_2|}
t^{(n-m)/2}\iiint e^{-||\eta||_{g(x)}^2}\eta^{\alpha_1}\\
&&\qquad\times e^{-||x-\tilde
x||_{g(x)}^2/(4t)}\left(\textstyle\frac{X-\tilde X}{2\sqrt
t}\right)^{\alpha_2} \langle
r_{n,j,\alpha}(x,D_M)\phi(x),\rho(\tilde x)\rangle d\nu_\eta
d\nu_xd\nu_{\tilde x}\,.\nonumber
\end{eqnarray*}
The $d\nu_\eta$ integral is over the complex domain
$\eta\in\mathbb{R}+\frac12\sqrt{-1}\textstyle\frac{X-\tilde X}{\sqrt t}$.
But we can deform that domain back to the real domain $\eta\in\mathbb{R}$. Set
\begin{eqnarray*}
&&c_{\alpha_1,\alpha_2,j}:=(2\pi)^{-m}\frac1{j!}\frac{(\alpha_1+\alpha_2)!}{\alpha_1!\alpha_2!}
(-\sqrt{-1})^{|\alpha_2|}
\int\eta^{\alpha_1}e^{-|\eta|_{g^*(x)}^2}d\nu_\eta\text{ to express}\\
&&\beta_{n,j,\alpha}(t)=\sum_{\alpha_1+\alpha_2=\alpha}c_{\alpha_1,\alpha_2,j}t^{(n-m)/2}\\
&&\qquad\times \iint e^{-||x-\tilde
x||_{g(x)}^2/(4t)}\left(\textstyle\frac{X-\tilde X}{2\sqrt
t}\right)^{\alpha_2}\langle
r_{n,j,\alpha}(x,D_M)\phi(x),\rho(\tilde x)\rangle
d\nu_xd\nu_{\tilde x}\,.
\end{eqnarray*}
This sum ranges over $|\alpha_1|$ even as otherwise $c_{\alpha_1,\alpha_2}$ vanishes. Thus
$|\alpha_2|\equiv|\alpha|\equiv n\mod 2$. This reduces the proof to considering expressions of the form:
\begin{equation}\label{E4.k}
\begin{array}{l}
\displaystyle
f_{n,j,\alpha,\alpha_2}(t):=t^{(n-m)/2}\iint
e^{-||x-\tilde x||_{g(x)}^2/(4t)}(\textstyle\frac{X-\tilde
X}{\sqrt t})^{\alpha_2}\\
\displaystyle\qquad\times \langle r_{n,j,\alpha}(x,D_M)\phi(x),\rho(\tilde
x)\rangle d\nu_xd\nu_{\tilde x},\\
\displaystyle\text{ where }|\alpha_2|\equiv n\text{ mod }2\text{ and
}\operatorname{ord}(r_{n,j,\alpha}(x,D_M))=n\,.\vphantom{\vrule height 12pt}
\end{array}\end{equation}
In examining such integrals, the fact that we are in the vector
valued setting  plays no role and we assume henceforth that $V$
and $V^*$ are 1-dimensional and hence $r$, $\phi$, and $\rho$ are
scalar valued.

\subsection{The interior terms in Theorem~\ref{T1.1}}\label{S4.3}

We begin our analysis with:
\begin{lemma}\label{L4.4}
Let $\Phi\in L^1(\mathbb{R}^m)$, let $\rho\in
C^k(\mathbb{R}^m)$ have compact support in an open subset
$\mathcal{U}\subset\mathbb{R}^m$, and let $(X-\tilde
X)_i:=g_{ij}(x-\tilde x)^j$. Let
$$
F(t):=t^{(n-m)/2}\int_{\mathcal{U}}\int_{\mathcal{U}}
e^{-||x-\tilde x||^2_{g(x)}/(4t)}\left({\textstyle\frac{X-\tilde
X}{\sqrt t}}\right)^{\alpha_2} \langle\Phi(x),\rho(\tilde
x)\rangle d\nu_xd\nu_{\tilde x}\,.$$
There exist smooth coefficients
$c_{\sigma,\alpha_2,g}(x)$
so that as $t\downarrow0^+$,
$$\left|F(t)-\sum_{|\sigma|=0}^{k-1}
 t^{(n+|\sigma|)/2}\int_{\mathcal{U}} c_{\sigma,\alpha_2,g}(x)
\langle \Phi(x),\rho^{(\sigma)}(x)\rangle d\nu_x\right|\le Ct^k\,.$$
\end{lemma}

\begin{proof} We make the change of variables $\tilde x=x+u$ and dually $\tilde X=X+U$ where $U_i=g_{ij}u^j$ to express
$$
F(t)=t^{(n-m)/2}\iint
e^{-||u||_{g(x)}^2/(4t)}\textstyle\left(\frac{U}{\sqrt
t}\right)^{\alpha_2}\langle\Phi(x),\rho(x+u)\rangle
d\nu_ud\nu_x\,.
$$
The $d\nu_u$ integral decays exponentially for
$|u|>t^{1/4}$ so we may assume the $d\nu_u$
integral is localized
to $|u|<t^{1/4}$. For $u$ small, we use
Equation~(\ref{E4.a}) to express:
\begin{eqnarray*}
&&\rho(x+u)\sim\sum_{|\sigma|=0}^{k-1}\textstyle\frac1{\sigma!}
u^\sigma d_x^\sigma\rho(x)+O(|u|^k),\\
&&F(t)=t^{(n-m)/2}\sum_{|\sigma|\le
k-1}{\textstyle\frac1{\sigma!}}\iint e^{-||u||_{g(x)}^2/(4t)}\\
&&\qquad\times
\left\{ ({\textstyle\frac{U}{\sqrt t}})^{\alpha_2}
u^\sigma\langle \Phi(x),\rho^{(\sigma)}(x)\rangle+O(|u|^k)\right\}
d\nu_ud\nu_x\,.\nonumber
\end{eqnarray*}
We set $\tilde u=u/\sqrt t$ and $\tilde U=U/\sqrt{t}$ to express
$$
F(t)=\sum_{|\sigma|=0}^{k-1}
{\textstyle\frac1{\sigma!}}t^{(n+|\sigma|)/2}
\iint e^{-||\tilde u||_{g(x)}^2/4}\tilde U^{\alpha_2}\tilde u^\sigma\langle\Phi(x),\rho^{(\sigma)}(x)\rangle
d\nu_{\tilde u}d\nu_x+O(t^k)\,.
$$
The $d\nu_x$ integral remains an integral over $\mathcal{U}$.
But as $t\downarrow0$, the $d\nu_{\tilde u}$
integral expands to $\mathbb{R}^m$ and defines the coefficients
$c_{\sigma,\alpha_2}$.
\end{proof}

We apply Lemma~\ref{L4.4} to the case
$\Phi=r_{n,j,\alpha}\phi$ in Equation~(\ref{E4.k}). By
assumption $r_{n,j,\alpha}$ is of order $n$ in the derivatives of
the total symbol of $D_M$. We have $\rho^{(\sigma)}$ is of order
$|\sigma|$ in the derivatives of $\rho$. Thus we have expressions
which are of order $n+|\sigma|$ in the derivatives of the symbol
of $D_M$ and in the derivatives of $\rho$. Furthermore, the
$d\nu_{\tilde u}$ integral vanishes unless $|\sigma|+|\alpha_2|$
is even. Since $|\alpha_2|\equiv n$ mod $2$, this implies
$|\sigma|+n$ is even so terms involving fractional powers of $t$
vanish as claimed. This leads to exactly the sort of interior
expansion described in Theorem~\ref{T1.1}.

\subsection{The proof of Theorem~\ref{T1.1}}\label{S4.4}
We now return to the general setting and deal with the case in which
 the coordinate chart meets the boundary.
We set $x=(r,y)$; the $d\nu_r$ integral ranges over $0\le r<\infty$
and the $d\nu_y$ integral ranges over $y\in\mathbb{R}^{m-1}$.
The $dy$ and $d\tilde y$ integrals are handled using the analysis of Lemma~\ref{L4.4}.
We therefore suppress these variables and
concentrate on the $d\nu_r$ integrals and in essence assume
that we are dealing with a $1$-dimensional problem;
we can always choose the coordinates so $ds^2=dr^2+g_{ab}(r,y)dy^idy^j$. We resume
the computation with Equation~(\ref{E4.k}) where we do not perform
the integrals in the two variables normal to the boundary.
We suppress other elements of the notation to examine an integral of the form:
$$
f(t):=t^{(n-1)/2}
\int_{0}^\infty\int_{0}^\infty  e^{-||x-\tilde
x||_e^2/(4t)}(\textstyle\frac{X-\tilde X}{\sqrt t})^{\alpha_2}
\langle r_{n,j,\alpha}(x,D_M)\phi(x),\rho(\tilde x)\rangle
d\nu_xd\nu_{\tilde x}\,.
$$
Here the integral has compact support in $(x,\tilde x)$
and is homogeneous
of degree $n$ in the derivatives of the symbol of $D_M$, in the
derivatives of $\phi$, and in the derivatives of $\rho$. We suppress the role of $|_{g}$ in
the tangential integrals which can also depend on the normal
parameter. We may use the
binomial theorem to expand $(X-\tilde X)^{\alpha_2}$
and absorb the relevant powers of $x$ and $\tilde x$ into
$\phi$ and $\rho$; this alters $a$ and $b$ appropriately and
the powers of $\sqrt t$ reflect this. We also replace $\phi$ by
$r_{n,j,\alpha}\phi$. We may therefore take
$\alpha_2=0$ and $r_{n,j,\alpha}=1$ when establishing the existence
of the appropriate asymptotic series.

To simplify the discussion, we may assume that $\phi$ and $\rho$
have their support contained in $[0,1]$ and vanish identically near $x=1$.
Let $\Xi$ be a monotonically decreasing smooth
cut-off function of compact support which is identically $1$ on the support of
$\phi$ and of $\rho$; $\Xi(x)$ is identically $1$ near $x=0$
and identically $0$ near $x=1$. We may use Taylor's theorem to
express
$$\phi(x)=\sum_{j=0}^N\phi_jx^{j-a}+\Phi\text{ and }
    \rho(x)=\sum_{j=0}^N\rho_jx^{j-b}+\tilde\Phi
$$
where $\Phi$ and $\tilde\Phi$ are $C^{k}$ and
vanish to order $k$ at the origin $0$ where $k$ becomes
arbitrarily large as $N\rightarrow\infty$.
This reduces the problem to considering integrals of the form:
\begin{eqnarray*}
&&f_{u,v,\Xi}(t)=\int_{0}^1\int_{0}^1
e^{-|x-\tilde x|^2/(4t)}\Xi(x)\Xi(\tilde x)x^{-u}\tilde x^{-v}d\tilde xdx,\\
&&f_{u,\tilde\Phi,\Xi}(t):=\int_{0}^1\int_{0}^1
e^{-|x-\tilde x|^2/(4t)}\Xi(x)\Xi(\tilde x)x^{-u}\tilde\Phi(\tilde x)
d\tilde xdx,\\
&&f_{\Phi,\tilde\Phi,\Xi}(t):=\int_{0}^1\int_{0}^1
e^{-|x-\tilde x|^2/(4t)}\Xi(x)\Xi(\tilde x)\Phi(x)\tilde\Phi(\tilde x)
d\tilde xdx\,,
\end{eqnarray*}
where $(u,v)\in\mathcal{O}$, where
$\Phi\in L^1$, where $\tilde\Phi$ is $C^k$, where $\tilde\Phi$ vanishes near $x=1$,
and where $\tilde\Phi$ vanishes to order $k$ at $x=0$.
We use Corollary~\ref{L3.11} to
show that the functions $f_{u,\tilde\Phi,\Xi}(t)$ have an appropriate asymptotic series.
Extend $\Phi$ and $\tilde\Phi$ to
be zero on $[-1,0]$; $\Phi$ is still in $L^1$ and $\tilde\Phi$ is in $C^k$.
Express
\begin{eqnarray*}
&&f_{u,\tilde\Phi,\Xi}(t)=\int_{\mathbb{R}}\int_{\mathbb{R}}
e^{-|x-\tilde x|^2/(4t)}\Xi(x)\Xi(\tilde x)x^{-u}\tilde\Phi(\tilde x)
d\tilde xdx,\\
&&f_{\Phi,\tilde\Phi,\Xi}(t)=\int_{\mathbb{R}}\int_{\mathbb{R}}
e^{-|x-\tilde x|^2/(4t)}\Xi(x)\Xi(\tilde x)\Phi(x)\tilde\Phi(\tilde x)
d\tilde xdx
\end{eqnarray*}
and the desired asymptotic series follows from Lemma~\ref{L4.4}. This completes the proof of
Theorem~\ref{T1.1}
\hfill\qed

\subsection{Logarithmic terms if $a+b=1-2k$}\label{S4.5}
Theorem~\ref{T1.1} has dealt with the case $a+b\ne1-2k$. The following is an immediate
consequence of Lemma~\ref{L3.14} and the arguments we have given above.
\begin{theorem}\label{T4.4} Let $(a,b)\in\mathbb{R}^2$ satisfy $\Re(a)<1$, $\Re(b)<1$, and $a+b=1-2k$ where $k$ is
a non-negative integer. Let $D_M$
be an operator of Laplace type on a smooth vector bundle $V$
over a compact Riemannian manifold $(M,g)$ without boundary.
Let $\Omega$ be a compact subdomain of $M$ with smooth
boundary.
Let $\phi\in C^\infty(V|_{\operatorname{int}(\Omega)})$ and let
$\rho\in C^\infty(V^*|_{\operatorname{int}(\Omega)})$.
We assume that $r^a\phi$ and $r^b\rho$ are smooth near the boundary of $\Omega$.
Let $\beta_\Omega(\phi,\rho,D_M)(t)$ be the
heat content of $\Omega$ in $M$. Then there is a complete
asymptotic expansion of $\beta_\Omega(\phi,\rho,D_M)(t)$
for small time such that for any positive integer $N>2k+3$ so that as $t\downarrow0$:
\begin{eqnarray*}
\beta_\Omega(\phi,\rho,D_M)(t)&=&\sum_{n=0}^{2N+2k}t^{n/2}\beta_{n,a,b}(\phi,\rho,D_M)
+\sum_{\ell=0}^Nt^{\ell+k}\log(t)\tilde\beta_{\ell,a,b}(\phi,\rho,D_M)
+O(t^{N+k})\,.
\end{eqnarray*}
The coefficient $\tilde\beta_{\ell,a,b}(\phi,\rho,D_M)$ of $t^k\log(t)$ for $\ell=0$
is given by
\begin{eqnarray*}
&&\tilde\beta_{0,a,b}^{\partial\Omega}(\phi,\rho,D_M)=
-\frac{a(a+1)\dots(a+2(k+1)-1)}{2\cdot k!}
\int_{\partial\Omega}\langle\phi_0,\rho_0\rangle dy\,.
\end{eqnarray*}
More generally, the remaining coefficients are locally computable as suitable regularized integrals over $M$
and integrals over the boundary.
\end{theorem}

\section{Heat content asymptotics with Neumann and Dirichlet boundary conditions}\label{S5}
Let $M$ be a compact smooth $m$ dimensional manifold with smooth boundary $\partial M$. Let
 $e^{-t\Delta_\pm}$ be the fundamental solution of the heat equation for the
Neumann $(+)$ or the Dirichlet $(-)$ realization of
 the Laplace-Beltrami operator. If $\phi,\rho\in L^1(M)$, then we may define:
 $$\beta_\pm^M(\phi,\rho)(t):=\int_Me^{-t\Delta_\pm}\phi\cdot\rho dx\,.$$
 In a subsequent paper, we plan to investigate the general setting; this will involve extending the
 analysis of Section~\ref{S4} to examine elliptic boundary conditions. For the moment, however,
 in the interests of brevity,
 we content  ourselves with establishing an improved version of Conjecture~1.2 of \cite{BGK12}
 in the 1-dimensional setting where $M=[0,1]$ and where
$\Delta=-\partial_x^2$:

\begin{theorem}\label{T5.1}
Let $\Re(a)<1$ and $\Re(b)<1$.
Let $\phi$ and $\rho$ be smooth on $(0,1)$. Assume $r^{a}\phi$ and $r^{b}\rho$ are smooth near the boundary of $M=[0,1]$. Let $\Delta=-\partial_x^2$.
\begin{enumerate}
\item If $a+b\ne1,-1,-3,\dots$,
then
there is a complete asymptotic series as $t\downarrow0$ with locally computable coefficients of the form:
\begin{eqnarray*}
\beta_\pm^\IN(\phi,\rho)(t)&\sim&
\sum_{n=0}^\infty t^n\beta_{\pm,n}(\phi,\rho)
+\sum_{j=0}^\infty t^{(1+j-a-b)/2}
   \beta_{\pm,j,a,b}(\phi,\rho)\,.
\end{eqnarray*}
\item If $a$ is real and if $a+b=-2k-1$ is an odd integer, then there is a complete
asymptotic series as $t\downarrow0$ with locally computable coefficients of the form:
\begin{eqnarray*}
\beta_\pm^\IN(\phi,\rho)(t)&\sim&\sum_{n=0}^\infty t^{n/2}\beta_{\pm,n,a,b}(\phi,\rho)
+\sum_{\ell=0}^\infty t^{\ell+k}\log(t)\tilde\beta_{\pm,\ell,a,b}(\phi,\rho)
\,.
\end{eqnarray*}
\end{enumerate}\end{theorem}

\begin{remark}\rm We can take Dirichlet boundary conditions at one end and Neumann boundary
conditions at the other end of the interval and obtain appropriate asymptotic series.
We can also let the growth and decay rates differ at the two components and again obtain appropriate
asymptotic series.
\end{remark}

\subsection{The half line}
Before proving Theorem~\ref{T5.1}, we must examine the
heat content asymptotics on the half-line.
Let $\phi\in L^1(\HL)$. Set
\begin{eqnarray*}
&&T\phi(x;t):=\frac1{\sqrt{4\pi t}}\int_0^\infty e^{-(x+\tilde x)^2/(4t)}\phi(\tilde x)d\tilde x,\quad\text{and}\\
&&H_\pm(\phi)(x;t):=\frac1{\sqrt{4\pi t}}\int_0^\infty\left\{e^{-(x-\tilde x)^2/(4t)}
\pm e^{-(x+\tilde x)^2/(4t)}\right\}\phi(\tilde x)d\tilde x\,.
\end{eqnarray*}
Let $e^{-t\Delta_+}$ and $e^{-t\Delta_-}$ be the Neumann $(+)$ and Dirichlet realizations $(-)$
of the Laplacian $\Delta=-\partial_x^2$ on the half line. The following result is well known. We shall
give the proof in the interests of completeness and to establish notation.

\begin{lemma}\label{L5.2}
Let $\phi\in L^1(\HL)$.
\begin{enumerate}
\item $||T\phi(\cdot;t)||_{L^1}\le||\phi||_{L^1}$.
\item If $\delta>0$, then
$||T(\chi_{\HU{\delta}{\infty}}\phi)(\cdot;t)||_{L^1}\le
2e^{-\delta^2/(8t)}||\phi||_{L^1}$.
\item $\lim_{t\downarrow0}||T\phi(\cdot;t)||_{L^1}=0$.
\item  $(\partial_t+\Delta)T\phi=0$.
\item $\left\{e^{-t\Delta\pm}\phi\right\}(x;t)=H_\pm(\phi)(x;t)$
\end{enumerate}
\end{lemma}

\begin{proof}
We establish Assertion~(1) by estimating:
\begin{eqnarray*}
&&||T\phi||_{L^1}\le\frac1{\sqrt{4\pi t}}\int_0^\infty\int_0^\infty
 e^{-(x+\tilde x)^2/(4t)}|\phi(\tilde x)|d\tilde xdx
\le\frac1{\sqrt{4\pi t}}\int_0^\infty\int_0^\infty e^{-x^2/(4t)}|\phi(\tilde x)|d\tilde xdx\\
&=&\frac1{\sqrt{4\pi t}}\int_0^\infty e^{-x^2/(4t)}dx\cdot\int_0^\infty|\phi(\tilde x)|d\tilde x
\le||\phi||_{L^1}\,.
 \end{eqnarray*}

We prove Assertion~(2) by computing similarly that:
\begin{eqnarray*}
&&||T(\chi_{\HU{\delta}{\infty}}\phi)(\cdot;t)||_{L^1}\le
\frac1{\sqrt{4\pi t}}\int_0^\infty\int_\delta^\infty
 e^{-(x+\tilde x)^2/(4t)}|\phi(\tilde x)|d\tilde xdx\\
 &\le&\frac1{\sqrt{4\pi t}}\int_0^\infty\int_\delta^\infty
 e^{-(x+\delta)^2/(4t)}|\phi(\tilde x)|d\tilde xdx\\
&\le&\frac1{\sqrt{4\pi t}}e^{-\delta^2/(8t)}\int_0^\infty e^{-x^2/(8t)}dx\cdot\int_0^\infty|\phi(\tilde x)|d\tilde x
\le2e^{-\delta^2/(8t)}||\phi||_{L^1}\,.
\end{eqnarray*}

Let $\epsilon>0$ be given. Choose $\delta>0$ so
$||\chi_{[0,\delta]}\phi||_{L^1}<\frac\epsilon2$. Choose $t_0$ so
$0<t<t_0$ implies $2e^{-\delta^2/(8t)}||\phi||_{L^1}<\frac\epsilon2$.
Assertion~(3) now follows from Assertion~(4).
To show that $(\partial_t+\Delta)T\phi=0$, we compute:
\begin{eqnarray*}
&&\displaystyle\partial_tT\phi(x;t)=\frac1{\sqrt{4\pi t}}\int_0^\infty\left\{-\frac1{2t}
+\frac{(x+\tilde x)^2}{4t^2}\right\}e^{-(x+\tilde x)^2/(4t)}\phi(\tilde x)dx,\\
&&\displaystyle\partial_x^2T\phi(x;t)=\frac1{\sqrt{4\pi t}}\int_0^\infty\left\{-\frac2{4t}+\frac{4(x+\tilde x)^2}{(4t)^2}\right\}
e^{-(x+\tilde x)^2/(4t)}\phi(\tilde x)dx,\\
&&(\partial_t+\Delta)T\phi(x;t)=(\partial_t-\partial_x^2)H(x;t)=0\,.
\end{eqnarray*}
We use Assertion~(3) and Assertion~(4) to see $(\partial_t+\Delta)H_\pm(\phi)(x;t)=0$
and $\lim_{t\downarrow0}H_\pm(\phi)(\cdot;t)=\phi(\cdot)$ in $L^1$. We complete
the proof of the Lemma by checking that the boundary conditions are satisfied:
\medbreak\quad
$\displaystyle\left\{e^{-t\Delta-}\phi\right\}(0;t)=\frac1{\sqrt{4\pi t}}\int_0^\infty
\left\{e^{-(0-\tilde x)^2/(4t)}
-e^{-(0+\tilde x)^2/(4t)}\right\}\phi(\tilde x)d\tilde x=0$,
\medbreak\quad
$\displaystyle\partial_x\left\{e^{-t\Delta+}\phi\right\}(0;t)
=\frac1{\sqrt{4\pi t}}\int_0^\infty\frac2{4t}
\left\{\tilde xe^{-(0-\tilde x)^2/(4t)}
-\tilde xe^{-(0+\tilde x)^2/(4t)}\right\}\phi(\tilde x)d\tilde x=0$.
\end{proof}
\subsection{The proof of Theorem~\ref{T5.1}}
The critical case to examine is $\phi(x)=x^{-a}$ and $\rho(x)=x^{-b}$; the remainder
of the analysis of the general case then follows using exactly the
same arguments using cut-off functions as was done previously. And thus all that is necessary
to do is to examine the correction term
$$
\tilde\beta(a,b;t):=\frac1{\sqrt{4\pi t}}
\int_0^1\int_0^1e^{-(x+\tilde x)^2/(4t)}x^{-a}\tilde x^{-b}dxd\tilde x
$$
as then Theorem~\ref{T5.1} will follow from our previous results.
We can replace $[0,1]$ by $\HL$ modulo an exponentially suppressed error term as
$t\downarrow0$. This expresses
\begin{eqnarray*}
\tilde\beta(a,b;t)&\sim&\frac1{\sqrt{4\pi t}}\int_0^\infty\int_0^\infty
e^{-(x+\tilde x)^2/(4t)}x^{-a}\tilde x^{-b}dxd\tilde x\\
&=&t^{(1-a-b)/2}\int_0^\infty\int_0^\infty
e^{-(x+\tilde x)^2}x^{-a}\tilde x^{-b}dxd\tilde x\,.
\end{eqnarray*}
This integral can
 be evaluated using the techniques in \cite{vdBGK} (see the proof of Lemma 1.6)
to yield the following formula for the correction term:
\medbreak\quad
$\displaystyle\tilde\beta(a,b;t)\sim t^{(1-a-b)/2}\cdot2^{-a -b}\pi^{-1/2}
\Gamma\left(\frac{2-a-b}2\right) \cdot
 \frac{\Gamma(1-a)\Gamma(1-b)}
 {\Gamma(2-a-b)}$.
  \hfill\qed

\subsection*{Acknowledgment}Research of P. Gilkey partially supported by project
Project MTM2009-07756 (Spain).

\end{document}